\documentclass[a4paper,abstracton]{scrartcl}
\usepackage{amsfonts}
\usepackage{amssymb}
\usepackage{amsmath}
\usepackage{amsthm}

\usepackage[ruled,vlined,linesnumbered]{algorithm2e}

\usepackage{graphicx}
\usepackage{color}

\usepackage{subfigure}

\usepackage{a4wide}

\usepackage[]{todonotes}

\newtheorem{theorem}{Theorem}
\newtheorem{observation}[theorem]{Observation}
\newtheorem{lemma}[theorem]{Lemma}

\newtheorem{cor}[theorem]{Corollary}
\newtheorem{definition}[theorem]{Definition}

\usepackage{authblk}

\allowdisplaybreaks

\begin{document}

\newcommand{\X}{{\mathcal{X}}}
\newcommand{\cU}{{\mathcal{U}}}
\newcommand{\cI}{{\mathcal{I}}}
\newcommand{\cC}{{\mathcal{C}}}
\newcommand{\R}{\mathbb{R}}
\newcommand{\N}{\mathbb{N}}

\title{Two-Stage Robust Optimization Problems with Two-Stage Uncertainty}

\author[1]{Marc Goerigk\footnote{Corresponding author, email \texttt{marc.goerigk@uni-siegen.de}}}
\author[2,3]{Stefan Lendl}
\author[2]{Lasse Wulf}

\affil[1]{Network and Data Science Management, University of Siegen,\newline Unteres Schlo{\ss} 3, 57072 Siegen, Germany}	
\affil[2]{Institute of Discrete Mathematics, Graz University of Technology,\newline  Steyrergasse 30/II, 8010 Graz, Austria}
\affil[3]{Institute of Operations and Information Systems, University of Graz,\newline  Universit\"atsstra{\ss}e 15, 8010 Graz, Austria}

\date{}

\maketitle

\begin{abstract}
We consider two-stage robust optimization problems, which can be seen as games between a decision maker and an adversary. After the decision maker fixes part of the solution, the adversary chooses a scenario from a specified uncertainty set. Afterwards, the decision maker can react to this scenario by completing the partial first-stage solution to a full solution.

We extend this classic setting by adding another adversary stage after the second decision-maker stage, which results in min-max-min-max problems, thus pushing two-stage settings further towards more general multi-stage problems. We focus on budgeted uncertainty sets and consider both the continuous and discrete case. For the former, we show that a wide range of robust combinatorial optimization problems can be decomposed into polynomially many subproblems, which can be solved in polynomial time for example in the case of (\textsc{representative}) \textsc{selection}. For the latter, we prove NP-hardness for a wide range of problems, but note that the special case where first- and second-stage adversarial costs are equal can remain solvable in polynomial time. 
\end{abstract}



\section{Introduction}\label{sec:intro}

As many real-world decision problems are affected by uncertainty, robust optimization has been developed as a methodology to incorporate knowledge of the uncertainty proactively. For some general constraint
\[ f(\pmb{x},\pmb{c}) \le b \]
with decision variables $\pmb{x}$ and parameter vector $\pmb{c}$, we assume that $\pmb{c}$ is uncertain and comes from an uncertainty set $\cU$. The classic robust optimization approach (see \cite{ben2009robust}) is to reformulate the uncertain constraint to the robust counterpart
\[ \max_{\pmb{c}\in\cU} f(\pmb{x},\pmb{c}) \le b.  \]
This has the disadvantage that one solution $\pmb{x}$ needs to be feasible for all possible choices $\pmb{c}$ simultaneously, while it is often possible in practice to adjust the solution after the uncertainty has been revealed \cite{ben2004adjustable}. By splitting the decision variables into $\pmb{x}$ (the here-and-now part) and $\pmb{y}$ (the wait-and-see part, which can be decided with knowledge of the scenario), the two-stage adjustable problem becomes
\[ \max_{\pmb{c}\in\cU} \min_{\pmb{y}\in\X(\pmb{x})} f(\pmb{x},\pmb{c},\pmb{y}) \le b, \]
where $\X(\pmb{x})$ contains additional constraints on $\pmb{y}$ that may depend on $\pmb{x}$.

While this approach is general, the analysis of such problems can be strengthened if we assume more structure in the set of feasible solutions and in the function $f$. In particular, we may consider the case where some combinatorial optimization problem is given, and where the objective functions $f(\pmb x, \pmb c)$ and $f(\pmb x, \pmb c, \pmb y)$ are linear in $\pmb{x}$ and $\pmb{y}$. 
That is, for some combinatorial optimization problem, where $\X \subseteq\{0,1\}^n$ is the set of its feasible solutions, the classic \emph{one-stage} problem is of the form
\[ \min_{\pmb{x}\in\X} \max_{\pmb{c}\in\cU} \pmb{c}^t\pmb{x}. \]
Again, one can split the decision variable into a here-and-now part and a wait-and-see part to obtain the following problem, which is known in the literature under the name \emph{two-stage adjustable problem}, or \emph{two-stage problem} for short. For sets of feasible solutions $\X,\X(\pmb{x})\subseteq\{0,1\}^n$, the problem is to find
\[ \min_{\pmb{x}\in\X} \max_{\pmb{d}\in\cU} \min_{\pmb{y}\in\X(\pmb{x})} \pmb{c}^t\pmb{x} + \pmb{d}^t\pmb{y}.\]

 This problem has been considered extensively, see e.g., \cite{chassein2018recoverable,kasperski2017robust,kasperski2017robustb} or the general survey \cite{kasperski2016robust}. However, note that there is a substantial difference between the one-stage and the two-stage problem: In the one-stage problem, the uncertainty affects the here-and-now variables (that is, the vectors $\pmb c$ and $\pmb x$ are multiplied with each other). In the two-stage problem on the other hand, the uncertainty affects the wait-and-see variables (that is, the vectors $\pmb d$ and $\pmb y$ are multiplied with each other). One could also say that in the one-stage problem the uncertainty affects the past, while in the two-stage problem the uncertainty affects the future. This substantial difference is also the reason that the one-stage problem in general is not a special case of the two-stage problem, even though the naming might suggest this.
 
\paragraph{One step further towards multi-stage problems.} We propose a variant of  the two-stage problem, which indeed contains the one-stage problem as a special case. In this variant the uncertainty is split into two parts, such that the first part affects the here-and-now variables, and the second part affects the wait-and-see variables (that is, both parts of the uncertainty affect past decisions). We call this variant \emph{two-stage robust optimization with two-stage uncertainty}. Formally, the problem is to find
\[ \min_{\pmb{x}\in\X} \max_{\pmb{c}\in\cU} \min_{\pmb{y}\in\X(\pmb{x})} \max_{\pmb{d}\in\cU(\pmb{c})} \pmb{c}^t \pmb{x} + \pmb{d}^t\pmb{y}, \]
where $\cU(\pmb{c})$ is the second-stage uncertainty set that may depend on the first-stage uncertainty choice $\pmb{c}$.

Compared to the classical two-stage problem, our variant includes an additional stage after the second decision of the decision-maker.  Problems of this type go one step further towards general multi-stage problems (see, e.g., \cite{bertsimas2016multistage,goerigk2020multistage}). They are also closely related to the vibrant field of explorable uncertainty \cite{feder2000computing,feder2007computing,erlebach2008computing,megow2017randomization,durr2018scheduling}, which has been presented using various framing, such as queries \cite{goerigk2015robust,halldorsson2021query}, probes \cite{goel2006asking,guha2007model,lee2022robust}, testing \cite{durr2020adversarial}, or information discovery \cite{vayanos2020robust}. What these approaches have in common is that the decision maker plays an active role in receiving information about which scenario will be chosen from the uncertainty set. Such an effect also occurs in our model, where the first-stage choice of items forces the adversary to reveal their costs. Different to most such models, the decision maker is forced to pack the items that have been queried, and the process is restricted to a single stage of queries.

We furthermore remark that it is a very natural problem variant: One could imagine for example a business owner, who has two consecutive business periods to deliver a promised product to her customer. At the start of every business period a decision has to be made, such that the sum of both decisions results in a finished product (that is, if in the first business period only a small portion of the product is produced, then in the second business period, all of the remaining product must be produced). However, each business period is subject to uncertainties, which are not known at the start of the business period. In this example, uncertainty affects the past, hence we are in a situation where the classical two-stage problem is not applicable.

While in general many different choices for the uncertainty sets $\mathcal{U}$ and $\mathcal{U}(\pmb c)$ are possible, the results in this paper mainly focus on so-called budgeted uncertainty sets as introduced by Bertsimas and Sim~\cite{bertsimas2003robust,bertsimas2004price}. In this setting, the adversary can choose up to $\Gamma$ items where costs become increased. Uncertainty sets of this type have seen frequent application to a diverse set of applications, such as wine grape harvesting \cite{bohle2010robust}, load planning of trains \cite{bruns2014robust}, inventory control \cite{bertsimas2006robust}, and evacuation planning \cite{goerigk2015two}.
We extend such uncertainty sets to our new setting.
Recall that the uncertainty in our new setting is split into two parts. Therefore, we consider the case where the adversary can increase costs of items in the first stage and increase costs of items in the second stage, in such a way that the total number of cost increases in the first and second stage combined is at most $\Gamma$. This is clearly a natural generalization of the classical budgeted uncertainty to two stages. We consider both the cases where the increases in cost are discrete (the so-called  \emph{discrete budgeted uncertainty}) or continuous (the so-called  \emph{continuous budgeted uncertainty}).

\paragraph{Our contributions.} In Section~\ref{sec:definitions} we formally introduce two-stage uncertainty and provide an example and some easy observations. In Section~\ref{sec:cont} we consider the two-stage robust problem with two-stage uncertainty in combination with two-stage \emph{continuous} budgeted uncertainty. Our main result is that for any nominal problem where solving the linear relaxation gives an integer solution, we can decompose the robust problem into $O(n^{2})$ nominal problems and $O(n^3)$ two-stage problems of a special structure. This special structure (and different generalizations of it) has been considered in the literature before, under the names \emph{coordination problem} \cite{chassein2020complexity}, \emph{optimization with interaction costs} \cite{lendl2019combinatorial} and \emph{two-stage optimization under the expected value criterion} \cite{goerigk2020two}.  This reveals an unexpected connection between the two-stage robust problem with two-stage uncertainty and the coordination problem \cite{chassein2020complexity}, or the interaction cost problem \cite{lendl2019combinatorial}, or the two-stage expected value criterion problem \cite{goerigk2020two}.

 As a particular consequence, we obtain that if these problems can be solved in polynomial time, so can the robust problem. This implies that the robust counterparts of many typical combinatorial optimization problems such as \textsc{selection} remain solvable in polynomial time. 

Apart from the main result, we also discuss a variant of budgeted uncertainty sets, which is sometimes used in the robust optimization literature (see, e.g., \cite{nasrabadi2013robust,chassein2018recoverable}), and in which case the robust problem becomes easy as well.

We then consider the setting of two-stage \emph{discrete} budgeted uncertainty in Section~\ref{sec:disc}, where we show that the robust problem allows a compact mixed-integer programming formulation, but becomes NP-hard for a wide range of combinatorial optimization problems. We also show that in the special case where cost vectors do not differ between first and second adversarial stage (that is, $\pmb c = \pmb d)$, problems becomes easier.

In Section~\ref{sec:experiments}, we run computational experiments. These experiments assess the benefit of our approach in comparison to classic one-stage robust optimization. They also assess the running time of our approach.

 Our paper concludes with Section~\ref{sec:conclusions}, where we point out further open questions.

\section{Problem definition}
\label{sec:definitions}

In this section, we first formally define the two-stage robust optimization problem with two-stage uncertainty. We then formally define two-stage budgeted uncertainty sets. After that, we give an example of the concepts introduced. Finally, we conclude the section with some easy general observations. Throughout the whole paper, we use the notation $[n]=\{1,\ldots,n\}$ and write vectors in bold.

\subsection{Definition of the two-stage robust problem}
The starting point is some combinatorial optimization problem, which we call the \emph{nominal problem}. We assume that
$\X_\text{nom} \subseteq \{0,1\}^n$ is the set of feasible solutions to the nominal problem and that $\X_\text{nom}$ can be described in matrix form, i.e.\ $\X_\text{nom} = \{\pmb{x} \in \{0,1\}^n : A\pmb{x}\ge\pmb{b}\}$ for some matrix $A\in\mathbb{R}^{m\times n}$ and some right-hand-side $\pmb{b}\in\mathbb{R}^m$. For a given cost vector $\pmb{c}\in\mathbb{R}^n$, the nominal problem is hence to solve
\[ \textsc{Nom}(\pmb{c}) = \min_{\pmb{x} \in \X_\text{nom}} \pmb{c}^t \pmb{x}. \]
Throughout this paper, we will consider several classical combinatorial 
optimization problems. Recall that the \textsc{selection} problem is given by the set of feasible solutions
$\X_\text{nom} = \{ \pmb{x}\in\{0,1\}^n : \sum_{i\in[n]} x_i= p\}$ for some constant $p$.
A related problem, the \textsc{representative selection} problem, 
is defined by the set of feasible solutions $\X_\text{nom} = \{ \pmb{x}\in\{0,1\}^n : \sum_{i\in T_j} x_i=1, j\in[\ell]\}$ 
for an item partition $T_1\cup T_2 \cup \ldots \cup T_\ell = [n]$, i.e., we have $T_i \cap T_j = \emptyset$ for $i \neq j$. Also note that commonly studied graph problems like the \textsc{shortest path} problem, the \textsc{spanning tree} problem and 
the \textsc{assignment} problem (see~\cite{korte2011combinatorial} for definitions) fit into our framework of combinatorial optimization problems.

For every nominal combinatorial optimization problem, the corresponding two-stage problem variant is the following: In the first stage, a vector $\pmb x \in \{0,1\}^n$ can be selected without any restriction. However, in the second stage, a vector $\pmb y \in \{0,1\}^n$ must be selected, such that $\pmb x$ and $\pmb y$ together form a feasible solution (formally, this means that $\pmb x + \pmb y \leq \pmb 1$ and $\pmb x + \pmb y \in \X_\text{nom}$). The vector $\pmb x$ is called the first-stage solution, or the here-and-now part, and the vector $\pmb y$ is called the second-stage solution, or the wait-and-see part.

Summarizing the above, we have that
\[
\X = \{0,1\}^n
\]
is the set of feasible first-stage solutions, and
\[
\X(\pmb{x}) = \Big\{ \pmb{y}\in\{0,1\}^n : A(\pmb{x}+\pmb{y}) \ge \pmb{b},\ \pmb{x}+\pmb{y} \le \pmb{1} \Big\}.
\]
is the set of feasible second-stage solutions. Note that $\X(\pmb{0}) = \X_\text{nom}$ is the set of feasible solutions for the nominal problem. To treat the case that a first-stage solution $\pmb{x}\in\X$ is chosen such that $\X(\pmb{x})=\emptyset$, 
we define $\min \emptyset = \infty$. This means that the first-stage solution $\pmb x$ has infinite costs in this case. 

We now define our new two-stage problem variant, as well as several of its subproblems.
In the \emph{adversarial recourse problem}, we solve for given $\pmb{x}\in\X$, $\pmb{c}\in\cU$, and $\pmb{y}\in\X(\pmb{x})$ the problem
\[ \textsc{AdvRec}(\pmb{x},\pmb{c},\pmb{y}) = \max_{\pmb{d}\in\cU(\pmb{c})} \pmb{c}^t \pmb{x} + \pmb{d}^t\pmb{y}, \]
i.e., we only consider the last problem stage. One layer above this is the \emph{recourse problem}, where we solve
\[ \textsc{Rec}(\pmb{x},\pmb{c}) =  \min_{\pmb{y}\in\X(\pmb{x})} \max_{\pmb{d}\in\cU(\pmb{c})} \pmb{c}^t \pmb{x} + \pmb{d}^t\pmb{y}. \]
The \emph{adversarial problem} is to solve
\[ \textsc{Adv}(\pmb{x}) = \max_{\pmb{c}\in\cU} \min_{\pmb{y}\in\X(\pmb{x})} \max_{\pmb{d}\in\cU(\pmb{c})} \pmb{c}^t \pmb{x} + \pmb{d}^t\pmb{y} \]
and finally, the \emph{two-stage robust problem with two-stage uncertainty} is given as follows
\[ \textsc{Rob} = \min_{\pmb{x}\in\X} \max_{\pmb{c}\in\cU} \min_{\pmb{y}\in\X(\pmb{x})} \max_{\pmb{d}\in\cU(\pmb{c})} \pmb{c}^t \pmb{x} + \pmb{d}^t\pmb{y}. \label{eq:rob} \]
We sometimes simply refer to this as the robust problem, if the context is clear.

\subsection{Definition of two-stage budgeted uncertainty}
Budgeted uncertainty sets describe uncertain scenarios, in which the adversary can increase the cost of up to $\Gamma$ items, where the value $\Gamma$ is called the \emph{budget}. We assume that there are $n$ items in total. For every item $i=1,\dots,n$, we are given two numbers $\underline{c}_i$ and $\overline{c}_i$ with $0 \leq \underline{c}_i \leq \overline{c}_i$. The initial cost of the item is given by $\underline{c}_i$, but the adversary can spend one unit of budget to increase the cost to $\overline{c}_i$. In the case of continuous budgeted uncertainty, this cost increase can also be fractional.

We extend the notion of continuous budgeted uncertainty to two-stage uncertainty sets by assuming that the adversary has a fixed budget $\Gamma  \geq 0$ that can be distributed throughout the first and second stage. That is, in the first stage we use a standard set (see \cite{bertsimas2003robust}) of the form
\[
\cU = \left\{ \pmb{c}\in\mathbb{R}^n_+ : c_i = \underline{c}_i + (\overline{c}_i-\underline{c}_i)\delta_i,\ \delta_i\in[0,1]\ \forall i\in[n], \  \sum_{i\in[n]} \delta_i \le \Gamma \right\},
\]
while for the second adversarial stage, only the remaining budget can still be used. That is, given a first-stage cost vector $\pmb{c}$, the amount of the available budget that has already been used is $\sum_{i\in[n]} (c_i - \underline{c}_i)/(\overline{c}_i - \underline{c}_i)$, therefore the remaining budget is $\Gamma - \sum_{i\in[n]} (c_i - \underline{c}_i)/(\overline{c}_i - \underline{c}_i)$. The corresponding second-stage uncertainty set is 
\[
\cU(\pmb{c}) = \left\{ \pmb{d}\in\mathbb{R}^n_+ : d_i = \underline{d}_i + (\overline{d}_i-\underline{d}_i)\delta_i,\  \delta_i\in[0,1]\ \forall i\in[n], \  \sum_{i\in[n]} \delta_i \le \Gamma - \sum_{i\in[n]} \frac{c_i - \underline{c}_i}{\overline{c}_i - \underline{c}_i} \right\}.
\]

In summary, $\cU$ and $\cU(\pmb c)$ together enforce that the adversary can distribute a budget of $\Gamma$ over the first and second stage combined. Finally, \emph{discrete} budgeted two-stage uncertainty sets are defined the same way as $\cU$ and $\cU(\pmb c)$ with the only difference that we use $\delta_i \in \{0,1\}$ instead of $\delta_i \in [0,1]$ for both sets.

\subsection{Example}

We give an example of the concepts introduced earlier in this section. In this example, we have a two-stage robust \textsc{selection} problem with two-stage discrete budgeted uncertainty. The parameters for the \textsc{selection} problem are as follows: We have three items ($n = 3$). Two out of three items need to be packed $(p = 2)$. Furthermore, the adversary can increase the costs of only one item to its upper bound $(\Gamma = 1)$. Finally, every item has costs as described in Table~\ref{tab:ex}, such that the first-stage costs and second-stage costs are equal ($\underline{c}_i = \underline{d}_i$ and $\overline{c}_i = \overline{d}_i$ for all $i$).
\begin{table}[htb]
\begin{center}
\begin{tabular}{r|rrr}
$i$ & 1 & 2 & 3 \\
\hline
$\underline{c}_i$ & 3 & 1 & 4 \\
$\overline{c}_i$  & 7 & 10 & 5 
\end{tabular}
\end{center}
\caption{Example uncertain \textsc{selection} problem.}\label{tab:ex}
\end{table}

In the classic one-stage problem, an optimal solution is to pack items 1 and 3 with objective value $7+4=11$ (the adversary increases the cost of item 1). In the two-stage setting, it is possible to pack only item 1 in the first stage. If the adversary increases the costs of this item, we complete the solution by packing item 2 in the second stage with total costs $7+1=8$. If the adversary decides to keep the budget for the second stage, we then complete the solution by packing item 3 with total costs $3+5=8$. Hence, we save three units in comparison to the one-stage solution.

This example showcases an interesting fact: The possibility to wait-and-see in the two-stage setting offers an advantage to the decision maker compared to the one-stage setting. We will determine the magnitude of this advantage using computational experiments in Section~\ref{sec:experiments}.

\subsection{General observations}

In this subsection, we make two easy general observations about two-stage robust problems with two-stage uncertainty. One observation is that they indeed do contain the one-stage problem as a special case. The second observation is that the problem collapses to a classical one-stage problem in the case where $\cU$ and $\cU(\pmb c)$ are independent of each other.

\begin{observation}
The two-stage robust problem with two-stage uncertainty contains the robust one-stage problem as a special case.
\end{observation}
\begin{proof}
Assume that a one-stage problem $\min_{\pmb{x}\in\X_\text{nom}}\max_{\pmb{c}\in\cU}$ is given. By defining the costs of each scenario in $\cU(\pmb{c})$ to be sufficiently large (e.g., set $\cU(\pmb{c}) = \{\pmb{d}^{wc}\}$ with $d^{wc}_i > \max_{\pmb{c}\in\cU} \|\pmb{c}\|_1$), any optimal solution to the two-stage problem with two-stage uncertainty
\[\min_{\pmb{x}\in\X} \max_{\pmb{c}\in\cU} \min_{\pmb{y}\in\X(\pmb{x})} \max_{\pmb{d}\in\cU(\pmb{c})} \pmb{c}^t \pmb{x} + \pmb{d}^t\pmb{y} \]
cannot choose to buy an item in the second stage. Hence, both problems are equivalent.
\end{proof}

\begin{observation}
If the first-stage and second-stage uncertainty sets are independent in the sense that $\cU(\pmb c)$ is constant with respect to $\pmb c \in \cU$, then the two-stage robust problem with two-stage uncertainty is equivalent to a classic robust one-stage problem.
\end{observation}
\begin{proof}
Let $\cU' := \cU(\pmb c)$ be the constant value of $\cU(\pmb c)$. Using the fact that for fixed $\pmb{x} \in \X$, the term $\max_{\pmb c \in \cU} \pmb {c}^t \pmb x$ is constant, we can rewrite the two-stage robust problem with two-stage uncertainty the following way:
\begin{align*}
&\min_{\pmb{x}\in\X} \max_{\pmb{c}\in\cU} \min_{\pmb{y}\in\X(\pmb{x})} \max_{\pmb{d}\in\cU'} \pmb{c}^t \pmb{x} + \pmb{d}^t\pmb{y}\\
= &\min_{\pmb{x}\in\X} \left( \max_{\pmb{c}\in\cU} \pmb{c}^t \pmb{x} + \min_{\pmb{y}\in\X(\pmb{x})} \max_{\pmb{d}\in\cU'} \pmb{d}^t\pmb{y} \right)\\
= &\min_{\pmb{x}\in\X} \min_{\pmb{y}\in\X(\pmb{x})} \max_{\pmb{d}\in\cU'} \left( \pmb{d}^t\pmb{y} + \left( \max_{\pmb{c}\in\cU} \pmb{c}^t \pmb{x} \right) \right)\\
=& \min_{\pmb{x}\in\X, \pmb{y}\in\X(\pmb{x})} \max_{\pmb{c}\in\cU, \pmb{d}\in\cU'}  (\pmb{c}, \pmb d)^t (\pmb x, \pmb y)
\end{align*}
This corresponds to a classic robust one-stage problem where the set of feasible solutions is $\{(\pmb x,\pmb y) \mid \pmb x \in \X, \pmb y \in \X(\pmb x)\}$ and the uncertainty set is $\cU \times \cU'$.
\end{proof}

\section{Two-stage continuous budgeted uncertainty}
\label{sec:cont}

In this section, we consider two-stage continuous budgeted uncertainty. Our main result is a decomposition result similar to the decomposition result of \cite{bertsimas2003robust}. This decomposition result also reveals a surprising connection between two-stage continuous budgeted uncertainty and the \emph{coordination problem} \cite{chassein2020complexity}, or \emph{optimization with interaction costs} \cite{lendl2019combinatorial}, or \emph{two-stage optimization under the expected value criterion} \cite{goerigk2020two}. In Subsection~\ref{subsection:statement_main_result} we state the main result, which we then prove in Subsection~\ref{subsection:proof_main_result}. Finally, in Subsection~\ref{subsection:uvar} we consider a different variant of continuous budgeted uncertainty that is also sometimes used in the literature.

\subsection{Statement of the main result}
\label{subsection:statement_main_result}
Let $P = \{ \pmb{x}\in[0,1]^n : A\pmb{x} \ge \pmb{b} \}$ be the polyhedron of the LP-relaxation of the nominal problem.
In this section we make use of the assumption that we can use the LP-relaxation to find an optimal solution for the nominal problem, that is:
\begin{equation}
P \text{ is an integral polyhedron.} \tag{A}\label{ass}
\end{equation}
We now state the main result of this section. It says that we can decompose the robust problem into multiple nominal problems and multiple so-called problems of type (\ref{p8}), which are two-stage problems with a special structure.
\begin{definition}
A \emph{problem of type (\ref{p8})} is the following problem, where $\tilde{\pmb{a}},\tilde{\pmb{b}},\tilde{\pmb{c}}\in\mathbb{R}^n_+$ and $\tilde{v}\ge 0$ are constant coefficients:
\begin{subequations}
\label{p8}
\begin{align}
\min\ & \tilde{\pmb{a}}^t\pmb{x} + \tilde{\pmb{b}}^t\pmb{y}^{(1)} + \tilde{\pmb{c}}^t\pmb{y}^{(2)} + \tilde{v} \\
\text{s.t. } 
& A(\pmb{x}+\pmb{y}^{(1)}) \ge \pmb{b}  \\
& A(\pmb{x}+\pmb{y}^{(2)}) \ge \pmb{b}  \\
& \pmb{x} + \pmb{y}^{(1)} \le \pmb{1} \\
& \pmb{x} + \pmb{y}^{(2)} \le \pmb{1} \\
& x_i \in\{0,1\} & \forall i\in[n] \\
& y^{(1)}_i,y^{(2)}_i \in \{0,1\} & \forall i\in[n].
\end{align}
\end{subequations}
\end{definition}

\begin{theorem}\label{th:p}
Let a robust two-stage combinatorial optimization problem with two-stage continuous budgeted uncertainty be given. If assumption~\eqref{ass} holds, then a solution to the robust problem can be obtained by taking the minimum of the objective values of
\begin{itemize}
\item $O(n^{2})$ many problems of nominal type, and
\item $O(n^3)$ many problems of type (\ref{p8}).
\end{itemize}
In particular, if problem of type~\eqref{p8} can be solved in polynomial time, so can the robust problem.
\end{theorem}

We note the similarity of the classic result from \cite{bertsimas2003robust}, which states that robust problems with one-stage budgeted uncertainty can be decomposed into $O(n)$ nominal problems. By including an additional stage on the decision variables and an additional stage on the adversarial variables, the problem complexity is now increased both in the number and the type of subproblems that need to be considered.

We now explain how problem of type~\eqref{p8} relates to the existing literature. By setting $\pmb{z}^{(1)} = \pmb{x}+\pmb{y}^{(1)}$ and $\pmb{z}^{(2)} = \pmb{x}+\pmb{y}^{(2)}$, the problem becomes
\[ 
\min_{\pmb{z}^{(1)},\pmb{z}^{(2)}\in\X_{\text{nom}}} \sum_{i\in[n]} \tilde{b}_iz^{(1)}_i + \sum_{i\in[n]} \tilde{c}_i z^{(2)}_i + \sum_{i\in[n]} (\tilde{a}_i - \tilde{b}_i - \tilde{c}_i) z^{(1)}_iz^{(2)}_i.
\]
Note that a bilinear term was introduced to the objective by the substitution. This particular problem has been considered by other authors before: In fact, it is an instance of the coordination problem (with potentially negative interaction costs) \cite{chassein2020complexity}. It is also a diagonal combinatorial optimization problem with interaction costs \cite{lendl2019combinatorial}, see also \cite{iwamasa}.
Furthermore, it is a two-stage combinatorial optimization problem with two scenarios under the expected value criterion \cite{goerigk2020two}.
Based on what is already known about problems~\eqref{p8} in \cite{goerigk2020two}, 
we can conclude the following complexity results.

\begin{cor}
The robust two-stage counterparts with two-stage continuous budgeted uncertainty for the \textsc{representative selection} problem, the \textsc{selection} problem, and the \textsc{shortest path} problem in series-parallel graphs are solvable in polynomial time.
\end{cor}

Note that it is a currently open research problem to determine the complexity of other combinatorial optimization problems of type~\eqref{p8}. Finally, note that the constants in problems~\eqref{p8} are non-negative. Hence, an approximation guarantee for these subproblems translates to an approximation guarantee of the original problem.

\begin{cor}
Let a combinatorial optimization problem be given, for which assumption~\eqref{ass} holds. If the problem of type~\eqref{p8} is $\alpha$-approximable, then the robust two-stage combinatorial optimization problem with two-stage continuous budgeted uncertainty is $\alpha$-approximable as well.
\end{cor}

Using the results in \cite{goerigk2020two}, we can conclude the following consequence.

\begin{cor}
The robust two-stage counterpart with two-stage continuous budgeted uncertainty of the \textsc{spanning tree} problem is $O(\log n)$-approximable.
\end{cor}

\subsection{Proof of the main result}
\label{subsection:proof_main_result}

Let $P(\pmb{x})$ $ = \{\pmb y \in [0,1]^n : A(\pmb x + \pmb y) \geq \pmb b, \pmb x  + \pmb y \leq \pmb 1\}$ be the polyhedron of the LP-relaxation of $\X(\pmb{x})$.
We start with the observation that assumption (A) is equivalent to assuming that each $P(\pmb{x})$ is integral.
\begin{lemma}
Assumption~\eqref{ass} is equivalent to:
\begin{equation}
P(\pmb{x}) \text{ is an integral polyhedron for all } \pmb{x}\in\{0,1\}^n. \tag{A'}\label{avar}
\end{equation}
\end{lemma}
\begin{proof}
Note that $P(\pmb{0}) = P$, which means that \eqref{avar} immediately implies \eqref{ass}. We now prove the opposite direction.

We assume that \eqref{ass} holds. Let some $\pmb{y}\in P(\pmb{x})$ be a fractional extreme point of $P(\pmb{x})$ for some fixed $\pmb{x}\in\{0,1\}^n$. Consider the point $\pmb{z} = \pmb{y} + \pmb{x}$, which is contained in $P$. As $\pmb{z}$ is fractional, it is is not an extreme point of $P$. This means that there exist integral extreme points $\pmb{z}^{(1)},\ldots,\pmb{z}^{(K)}$ of $P$ and weights $\lambda_1,\ldots,\lambda_K$ with $\lambda_k \in (0,1)$, $\sum_{k\in[K]} \lambda_k = 1$ such that $\pmb{z} = \sum_{k\in[K]} \lambda_k \pmb{z}^{(k)}$. We claim that for every $k$, it holds that $\pmb{z}^{(k)}-\pmb x \in P(\pmb x)$. Indeed, it is immediately clear, that $A(\pmb x + (\pmb{z}^{(k)}-\pmb x)) \geq \pmb b$ and that $\pmb x  + (\pmb{z}^{(k)}-\pmb x) \leq \pmb 1$.
Furthermore, because $\pmb z  = \pmb x + \pmb y \leq \pmb 1$, it holds that $z_i = 1$ for all $i\in[n]$ where $x_i=1$. Therefore, we also have that $z^{(k)}_i=1$ for each such $i$ and each $k\in[K]$. This implies that $\pmb{z}^{(k)}-\pmb{x} \in [0, 1]^n$. In total, we have that $\pmb{z}^{(k)}-\pmb{x}$ is contained in $P(\pmb{x})$ for each $k=1,\dots,K$, and $\pmb{y} = \sum_{k\in[n]} \lambda_k(\pmb{z}^{(k)}-\pmb{x})$. This gives a contradiction, as $\pmb{y}$ was assumed to be an extreme point. We conclude that \eqref{avar} holds.
\end{proof}

To prove Theorem~\ref{th:p}, we first start by considering the adversarial recourse problem. For constant $\pmb x, \pmb c$ and $\pmb y$, the adversarial recourse problem $\textsc{AdvRec}(\pmb x, \pmb c, \pmb y)$ can be written as a linear program with respect to the variables $\pmb \delta$.
\begin{align*}
\textsc{AdvRec}(\pmb x, \pmb c, \pmb y)
= {\pmb{c}^t} \pmb{x} + \max \ &\sum_{i \in [n]}(\underline{d}_i + (\overline{d}_i-\underline{d}_i)\delta_i)y_i \\
\text{s.t. }
  &\sum_{i\in[n]} \delta_i \le \Gamma - \sum_{i\in[n]} \frac{c_i - \underline{c}_i}{\overline{c}_i - \underline{c}_i}\\  
  &0 \leq \delta_i \le 1 &\forall i\in [n]\\
\end{align*}
By introducing dual variables $\pi^{(2)}$ and $\rho_i^{(2)}$, this linear program can be dualized, yielding an equivalent formulation as a minimization problem. Integrating this minimization problem into the recourse problem using weak and strong duality, we reach the following problem formulation:
\begin{subequations}
\label{p1}
\begin{align}
 \textsc{Rec}(\pmb{x},\pmb{c}) = {\pmb c}^t \pmb x + \min\ & \sum_{i\in[n]} \underline{d}_i y_i + \left( \Gamma - \sum_{i\in[n]} \frac{c_i - \underline{c}_i}{\overline{c}_i - \underline{c}_i} \right) \pi^{(2)} + \sum_{i\in[n]} \rho^{(2)} _i \\
\text{s.t. } & A(\pmb{x}+\pmb{y}) \ge \pmb{b} \\
& \pmb{x}+\pmb{y} \le \pmb{1}  \\
& \pi^{(2)} + \rho^{(2)}_i \ge (\overline{d}_i - \underline{d}_i) y_i & \forall i\in[n] \label{eq:rho_i_2}\\
& y_i\in\{0,1\} & \forall i\in[n] \\
& \pi^{(2)} \ge 0 \\
& \rho^{(2)}_i \ge 0 & \forall i\in[n]. \label{eq:rho_i_2_dash}
\end{align}
\end{subequations}
Consider the case that variables $\pmb{y}$ are fixed. Then the following is a well-known argument (see, e.g., \cite{bertsimas2003robust}). The optimization problem has an optimal solution such that 
\[\rho_i^{(2)} = [(\overline{d}_i - \underline{d}_i)y_i - \pi^{(2)}]_+ =  [\overline{d}_i - \underline{d}_i - \pi^{(2)}]_+ y_i,\] where  $[x]_+$ denotes the positive part $\max\{x,0\}$. The first part of the equality follows from the fact that in an optimal solution $\rho^{(2)}_i$ is as small as possible while satisfying inequalities (\ref{eq:rho_i_2}) and (\ref{eq:rho_i_2_dash}). The second part follows from the fact that $y_i$ is binary.
For the remainder of this section, let 
\[\Pi := \{ \overline{d}_i - \underline{d}_i : i\in[n]\} \cup \{0\}. \] 
Note that with respect to the variable $\pi^{(2)}$, if all other variables are fixed, then the objective function is piecewise linear in $\pi^{(2)}$. Therefore the optimum is obtained at a point where the slope of this piecewise linear function changes. Because the set $\Pi$ is exactly the set of all these points, we have that $\pi^{(2)} \in \Pi$. Let $K$ be the cardinality of $\Pi$ and $\pi^{(2)}_1, \dots \pi^{(2)}_K$ be  all the elements of $\Pi$, that is 
\[\Pi = \{ \pi^{(2)}_1, \ldots, \pi^{(2)}_K\}. \]
Using this notation, problem~\eqref{p1} is equivalent to
\begin{equation}\label{p2}
{\pmb c}^t \pmb x + \min_{k\in[K]} \min_{\pmb{y}\in\X(\pmb{x})} 
\sum_{i\in[n]} (\underline{d}_i + [\overline{d}_i - \underline{d}_i - \pi^{(2)}_k]_+) y_i + \left( \Gamma - \sum_{i\in[n]} \frac{c_i - \underline{c}_i}{\overline{c}_i - \underline{c}_i} \right) \pi^{(2)}_k.
\end{equation}

\begin{lemma}
Problem~\eqref{p2} is an equivalent formulation of the recourse problem of the robust two-stage combinatorial optimization problem with two-stage continuous budgeted uncertainty.
\end{lemma}

Note that for each $k\in[K]$, problems~\eqref{p2} are problems where assumption~\eqref{avar} can be applied. In particular, minimizing a linear objective over $\X(\pmb x)$ is equivalent to minimizing a linear objective over its relaxation $P(\pmb x)$. We thus dualize the inner minimization problem, combine the resulting $K$ dual problems into a single problem, and integrate the decision variables and the constraints of the adversarial stage. This yields the following reformulation of the adversarial problem:
\begin{subequations}
\label{p3}
\begin{align}
\textsc{Adv}(\pmb{x},\pmb{c}) = \max\ & t + \sum_{i\in[n]} (\underline{c}_i + (\overline{c}_i-\underline{c}_i)\delta_i)x_i  \\
\text{s.t. } & t \le 
(\pmb{b}-A\pmb{x})^t \pmb{\alpha}^{(k)} + (\pmb{x}-\pmb{1})^t\pmb{\beta}^{(k)}
+(\Gamma - \sum_{i\in[n]} \delta_i) \pi^{(2)}_k\hspace{-2mm} & \forall k\in[K] \label{p3c1}\\
&  (A^t\pmb{\alpha}^{(k)})_i - \beta^{(k)}_i \le \underline{d}_i + [\overline{d}_i - \underline{d}_i - \pi^{(2)}_k]_+ & \forall i\in[n], k \in [K] \label{p3c2}\\
& \sum_{i\in[n]} \delta_i \le \Gamma \label{p3c3}\\
& \delta_i \in [0,1] & \forall i\in[n] \label{p3c4}\\
& \pmb{\alpha}^{(k)} \in\mathbb{R}^m_+ & \forall k\in[K] \\
& \pmb{\beta}^{(k)} \in\mathbb{R}^n_+ &\forall k\in[K].
\end{align}
\end{subequations}
Here, variables $\pmb{\alpha}^{(k)}$ are duals for the constraints $A\pmb{y} \ge \pmb{b} - A\pmb{x}$, while variables $\pmb{\beta}^{(k)}$ are duals for constraints $\pmb{y} \le \pmb{1}-\pmb{x}$. The variable $t$ and the inequalities containing it are used to replace the outer minimum in problem~\eqref{p2}. Variables $\pmb{\delta}$ are introduced to model where the adversary distributes the uncertainty of the first stage. Note that we can replace $(c_i - \underline{c}_i)/(\overline{c}_i - \underline{c}_i)$ by $\delta_i$.

\begin{lemma}
If assumption~\eqref{ass} holds, then problem~\eqref{p3} is an equivalent formulation of the adversarial problem of the robust two-stage combinatorial optimization problem with two-stage continuous budgeted uncertainty. 
\end{lemma}

Problem~\eqref{p3} is a linear program. We can hence dualize this problem, introducing dual variables $z^{(k)}, y_i^{(k)}, \pi^{(1)}$, and $\rho^{(1)}_i$ for the constraints \eqref{p3c1}, \eqref{p3c2}, \eqref{p3c3}, and \eqref{p3c4}, respectively. Furthermore, we add first-stage decision variables $\pmb{x}\in\{0,1\}^n$ to reach the following compact formulation of the two-stage robust problem with two-stage uncertainty \textsc{Rob}:
\begin{subequations}
\label{p4}
\begin{align}
\min\ & \sum_{i \in[n]} \underline{c}_i x_i + \sum_{k\in[K]} \sum_{i\in[n]} (\underline{d}_i + [\overline{d}_i - \underline{d}_i - \pi^{(2)}_k]_+) y^{(k)}_i \nonumber\\
& + \Gamma \sum_{k\in[K]} \pi^{(2)}_k z^{(k)} + \Gamma \pi^{(1)} + \sum_{i\in[n]} \rho^{(1)}_i \hspace{-3cm}\\
\text{s.t. } 
& A\pmb{y}^{(k)} \ge (\pmb{b}-A\pmb{x}) z^{(k)} & \forall k\in[K] \\
& (1-x_i)z^{(k)} \ge y^{(k)}_i & \forall i\in[n],k\in[K] \label{p4c2}\\
& \sum_{k\in[K]} \pi^{(2)}_k z^{(k)} + \pi^{(1)} + \rho^{(1)}_i \ge (\overline{c}_i-\underline{c}_i)x_i & \forall i\in[n] \\
& \sum_{k\in[K]} z^{(k)} = 1 \\
& x_i \in\{0,1\} & \forall i\in[n] \\
& y^{(k)}_i \ge 0 & \forall i\in[n], k\in[K] \\
& z^{(k)} \ge 0 & \forall k\in[K] \\
& \pi^{(1)} \ge 0 \\
& \rho^{(1)}_i \ge 0 & \forall i\in[n].
\end{align}
\end{subequations}
For the next step, a normalization of variable values will turn out to be helpful. Note that in any feasible solution, the inequality $y^{(k)}_i \leq z^{(k)}$ holds due to constraints \eqref{p4c2}. We introduce new variables $\tilde{y}^{(k)}_i$ such that the equation $y^{(k)}_i = z^{(k)}\tilde{y}^{(k)}_i$ holds and the new variables $\tilde{y}^{(k)}_i$ are bounded by 1. If $z^{(k)} \neq 0$, then the value of $\tilde{y}^{(k)}_i$ is uniquely defined by the above equation and we have $\tilde{y}^{(k)}_i \in [0,1]$.
If $z^{(k)}=0$, then the above equation trivially holds and we accept any choice for $\tilde{y}^{(k)}_i\in[0,1]$. We can hence substitute any occurrence of $y^{(k)}_i$ in problem \eqref{p4} by $z^{(k)}\tilde{y}^{(k)}_i$ in order to obtain the following equivalent problem:
\begin{subequations}
\label{p5}
\begin{align}
\min\ & \sum_{i \in[n]} \underline{c}_i x_i + \sum_{k\in[K]} \sum_{i\in[n]} (\underline{d}_i + [\overline{d}_i - \underline{d}_i - \pi^{(2)}_k]_+) z^{(k)}\tilde{y}^{(k)}_i \nonumber \\ 
&+ \Gamma \sum_{k\in[K]} \pi^{(2)}_k z^{(k)} + \Gamma \pi^{(1)} + \sum_{i\in[n]} \rho^{(1)}_i \hspace{-3cm} \\
\text{s.t. } 
& A(\pmb{x}+\tilde{\pmb{y}}^{(k)}) \ge \pmb{b} & \forall k\in[K] \\
& \pmb{x} + \tilde{\pmb{y}}^{(k)} \le \pmb{1}  & \forall k\in[K] \\
& \sum_{k\in[K]} \pi^{(2)}_k z^{(k)} + \pi^{(1)} + \rho^{(1)}_i \ge (\overline{c}_i-\underline{c}_i)x_i & \forall i\in[n] \\
& \sum_{k\in[K]} z^{(k)} = 1 \\
& x_i \in\{0,1\} & \forall i\in[n] \\
& \tilde{y}^{(k)}_i \in[0,1] & \forall k\in[K], i\in[n] \\
& z^{(k)} \ge 0 & \forall k\in[K] \\
& \pi^{(1)} \ge 0 \\
& \rho^{(1)}_i \ge 0 & \forall i\in[n].
\end{align}
\end{subequations}

\begin{lemma}
\label{lemma:robust-reformulation}
If assumption~\eqref{ass} holds, then problem~\eqref{p5} is an equivalent formulation of the robust two-stage combinatorial optimization problem with two-stage continuous budgeted uncertainty. 
\end{lemma}

As each variable $\rho^{(1)}_i$ occurs in only one constraint, we can write
\[\rho^{(1)}_i = [(\overline{c}_i  - \underline{c}_i)x_i - \pi^{(1)} - \sum_{k\in[K]} \pi^{(2)}_kz^{(k)}]_+\]
 due to an argument that is similar to another argument above. In the following lemma, we now want to prove that all optimal solutions to problem~\eqref{p5} have a special structure. Indeed, let us assume that all variables in problem~\eqref{p5} except $\pmb{z}$ are fixed. The remaining problem in $\pmb{z}$ is of the following form:
\begin{align*}
\min\ & \sum_{k\in[K]} (\Gamma\pi^{(2)}_k + \sum_{i\in[n]} (\underline{d}_i + [\overline{d}_i - \underline{d}_i - \pi^{(2)}_k]_+)\tilde{y}^{(k)}_i) z^{(k)} \\
& + \sum_{i\in[n]} [(\overline{c}_i  - \underline{c}_i)x_i - \pi^{(1)} - \sum_{k\in[K]} \pi^{(2)}_kz^{(k)}]_+ \\
\text{s.t. } & \sum_{k\in[K]} z^{(k)} = 1 \\
& z^{(k)} \ge 0 & \forall k\in[K].
\end{align*}

\begin{lemma}\label{lemmaz}
Let any $\pmb{a},\pmb{c}\in\mathbb{R}^k$ and $\pmb{b},\pmb{d}\in\mathbb{R}^n$ be given. For an optimization problem of the type
\begin{align*}
\min\ & \sum_{k\in[K]} a_k z^{(k)} + \sum_{i\in[n]} d_i [b_i - \sum_{k\in[K]} c_kz^{(k)}]_+ \\
\text{s.t. } & \sum_{k\in[K]} z^{(k)} = 1 \\
& z^{(k)} \ge 0 & \forall k\in[K],
\end{align*}
there is an optimal solution $\pmb{z}$ with
\begin{itemize}
\item $z^{(k_1)}=1$ for one $k_1$ and $z^{(k)}=0$ for all $k\neq k_1$, or
\item $z^{(k_1)}=(b_{i'}-c_{k_2})/(c_{k_1}-c_{k_2})\ge 0$, $z^{(k_2)}=(c_{k_1}-b_{i'})/(c_{k_1}-c_{k_2})\ge 0$ for some $k_1,k_2\in[K]$ and $i'\in[n]$, and $z^{(k)} = 0$ for all other $k \neq k_1, k_2$.
\end{itemize}
\end{lemma}
\begin{proof}
Let $\pmb{b}'=(b'_1,\ldots,b'_{\Delta})^t$ contain all unique values of $\pmb{b}$ in sorted order from smallest to largest. We write $b'_0 = -\infty$, $b'_{\Delta+1} = +\infty$. In each segment $\sum_{k\in[K]} c_k z^{(k)} \in [b'_{i'},b'_{i'+1}]$ for fixed $i'\in\{0,\ldots,\Delta\}$, 
the objective becomes linear. Hence, we guess the segment $i'$ of the optimal solution and find
\begin{align*}
\min\ & \sum_{k\in[K]} a_k z^{(k)} + \sum_{i : b_i > b'_{i'}} d_i (b_i - \sum_{k\in[K]} c_k z^{(k)}) \\
\text{s.t. } & \sum_{k\in[K]} z^{(k)} = 1 \\
& b'_{i'} \le \sum_{k\in[K]} c_k z^{(k)} \le b'_{i'+1} \\
& z^{(k)} \ge 0 & \forall k\in[K].
\end{align*}
We first consider the case that $i'\in\{1,\ldots,\Delta-1\}$. Adding two slack variables, we have a linear program with $K+2$ variables and 3 constraints. This means that in an optimal basis solution, there are at most three non-zero variables. We distinguish the following cases:
\begin{enumerate}
\item Two slack variables and one $z^{(k_1)}$ are non-zero. In this case, $z^{(k_1)} = 1$ as claimed.
\item One slack variable and two variables $z^{(k_1)}$ and $z^{(k_2)}$ are non-zero. This implies that the chosen optimal basis solution satisfies the following system of equations:
\begin{align*}
z^{(k_1)} + z^{(k_2)} &= 1 \\
c_{k_1} z^{(k_1)} + c_{k_2} z^{(k_2)} &= b'_{i'} \text{ or } b'_{i'+1}
\end{align*}
As this is a basis solution, we know that the columns $(1,c_{k_1})^t$ and $(1,c_{k_2})^t$ are linearly independent. Hence, $c_{k_1} \neq c_{k_2}$. Without loss of generality, we assume $c_{k_1} > c_{k_2}$, which gives $z^{(k_1)}=(b_{i'}-c_{k_2})/(c_{k_1}-c_{k_2})> 0$ and $z^{(k_2)}=(c_{k_1}-b_{i'})/(c_{k_1}-c_{k_2})> 0$.
\item If both slack variables are zero, we get $b'_{i'} = b'_{i'+1}$, which contradicts the assumption that $\pmb{b}'$-values are pairwise distinct.

\end{enumerate}
Finally, if $i' = 0$ or $i'=\Delta$, then the linear program has $K+1$ variables and 2 constraints. Again, the first two cases of the above case distinction are possible.
We conclude that there always exists an optimal solution to $\pmb{z}$ as claimed.
\end{proof}

Using Lemma~\ref{lemmaz}, we can rewrite problem~\eqref{p5} by enumerating possible values of $\pmb{z}$. Let us first assume that $z^{(k)}=1$ for some $k\in[K]$. Each remaining problem is then as follows:
\begin{align*}
\min\ & \sum_{i \in[n]} \underline{c}_i x_i + \sum_{i\in[n]} (\underline{d}_i + [\overline{d}_i - \underline{d}_i - \pi^{(2)}_k]_+) \tilde{y}_i + \Gamma \pi^{(2)}_k + \Gamma \pi^{(1)} + \sum_{i\in[n]} [\overline{c}_i-\underline{c}_i - \pi^{(1)} - \pi^{(2)}_k]_+ x_i \hspace{-2cm} \\
\text{s.t. } & A(\pmb{x}+\tilde{\pmb{y}}) \ge \pmb{b}  \\
& \pmb{x} + \tilde{\pmb{y}} \le \pmb{1} \\
& x_i \in\{0,1\} & \forall i\in[n] \\
& \tilde{y}_i \in[0,1] & \forall i\in[n] \\
& \pi^{(1)} \ge 0.
\end{align*}
Notice that the cost function is piecewise linear in variable $\pi^{(1)}$. We hence conclude that there is an optimal solution where $\pi^{(1)}\in\{\overline{c}_i - \underline{c}_i - \pi^{(2)}_k : i\in[n]\} \cup \{0\}$. Enumerating these $O(n)$ possible values for $\pi^{(1)}$, each subproblem is then of the form
\begin{subequations}
\label{p6}
\begin{align}
\min\ & \tilde{\pmb{c}}^t\pmb{x} + \tilde{\pmb{d}}^t \tilde{\pmb{y}} + \text{const.}\\
\text{s.t. } & A(\pmb{x}+\tilde{\pmb{y}}) \ge \pmb{b}  \\
& \pmb{x} + \tilde{\pmb{y}} \le \pmb{1} \\
& x_i \in\{0,1\} & \forall i\in[n] \\
& \tilde{y}_i \in[0,1] & \forall i\in[n] .
\end{align}
\end{subequations}
Note that for each fixed $\pmb x$, assumption \eqref{avar} can be applied, because the remaining problem is to optimize a linear objective over $P(\pmb x)$. We can therefore assume that $\tilde{y}_i \in \{0, 1\}$ as well. This leaves for every $i \in [n]$ with $(x_i, \tilde{y}_i) \neq (0, 0)$ only two choices: Either $x_i = 1$ and $\tilde{y}_i=0$, or $x_i = 0$ and $\tilde{y}_i=1$. For each $i\in[n]$, one can determine which choice is cheaper (with respect to the objective). Hence, problem~\eqref{p6} can be reduced to a nominal problem. In total, we have $O(n^{2})$ different nominal problems, as there are $O(n)$ choices for both $k$ and $\pi^{(1)}$.

In summary, in the first case of the main theorem, we have to solve the following nominal problems for each $k \in [K]$ and $\pi^{(1)}\in\{\overline{c}_i - \underline{c}_i - \pi^{(2)}_k : i\in[n]\} \cup \{0\}$: 
\begin{subequations}
\label{p6-dash}
\begin{align}
\min\ & \sum_{i \in [n]} \min\{c_i, d_i\}z_i  + M\\
\text{s.t. } & A \pmb z \ge \pmb{b}  \\
& z_i \in\{0,1\} & \forall i\in[n] \\
 \text{with constants }
&M = \Gamma \pi^{(2)}_k + \Gamma \pi^{(1)} \nonumber \\
&c_i = \underline{c}_i + [\overline{c}_i-\underline{c}_i - \pi^{(1)} - \pi^{(2)}_k]_+ \nonumber \\
&d_i = \underline{d}_i + [\overline{d}_i - \underline{d}_i - \pi^{(2)}_k]_+ \nonumber 
\end{align}
\end{subequations}

Let us now assume that two of the $\pmb{z}$-variables are active. Using Lemma~\ref{lemmaz}, we have
\begin{align*}
z^{(k_1)} &= ( (\overline{c}_{i'} - \underline{c}_{i'})x_{i'} - \pi^{(1)} - \pi^{(2)}_{k_2}) / (\pi^{(2)}_{k_1} - \pi^{(2)}_{k_2}) \\
z^{(k_2)} &= ( \pi^{(2)}_{k_1} - (\overline{c}_{i'} - \underline{c}_{i'})x_{i'} + \pi^{(1)}) / (\pi^{(2)}_{k_1} - \pi^{(2)}_{k_2}) \\
\pi^{(2)}_{k_1}z^{(k_1)} + \pi^{(2)}_{k_2} z^{(k_2)}& = (\overline{c}_{i'} - \underline{c}_{i'})x_{i'} - \pi^{(1)},
\end{align*}
where we only need to consider the case where both the variables $z^{(k_1)}$ and $z^{(k_2)}$ are active, that is we only need to consider combinations with $0 < z^{(k_1)},z^{(k_2)} < 1$.
Using these values in problem~\eqref{p5}, we find the following subproblems:
\begin{subequations}
\label{p7}
\begin{align}
\min\ & \sum_{i\in[n]} \underline{c}_i x_i \nonumber\\ 
& + \sum_{i\in[n]} (\underline{d}_i + [\overline{d}_i - \underline{d}_i - \pi^{(2)}_{k_1}]_+) \frac{(\overline{c}_{i'} - \underline{c}_{i'})x_{i'} - \pi^{(1)} - \pi^{(2)}_{k_2}}{\pi^{(2)}_{k_1} - \pi^{(2)}_{k_2}} \tilde{y}^{(k_1)}_i \hspace{-1cm} \nonumber\\
& + \sum_{i\in[n]} (\underline{d}_i + [\overline{d}_i - \underline{d}_i - \pi^{(2)}_{k_2}]_+) \frac{\pi^{(2)}_{k_1} - (\overline{c}_{i'} - \underline{c}_{i'})x_{i'} + \pi^{(1)}}{\pi^{(2)}_{k_1} - \pi^{(2)}_{k_2}} \tilde{y}^{(k_2)}_i \hspace{-1cm} \nonumber\\
&+ \Gamma (( \overline{c}_{i'} - \underline{c}_{i'})x_{i'} - \pi^{(1)} ) + \Gamma \pi^{(1)} + \sum_{i\in[n]} \rho^{(1)}_i \\
\text{s.t. } & 
A(\pmb{x}+\tilde{\pmb{y}}^{(k)}) \ge \pmb{b}  & \forall k\in\{k_1,k_2\} \\
& \pmb{x} + \tilde{\pmb{y}}^{(k)} \le \pmb{1} & \forall k\in\{k_1,k_2\} \\
&( \overline{c}_{i'} - \underline{c}_{i'})x_{i'} - \pi^{(1)} + \pi^{(1)} + \rho^{(1)}_i \ge (\overline{c}_i - \underline{c}_i) x_i & \forall i\in[n] \label{p7c}\\
& x_i \in \{0,1\} & \forall i\in[n] \\
& \tilde{y}^{(k)}_i \in [0,1] & \forall i\in[n], k\in\{k_1,k_2\} \\
& (\overline{c}_{i'} - \underline{c}_{i'})x_{i'} - \pi^{(1)} - \pi^{(2)}_{k_2} \geq 0 \label{p7g}\\
& \pi^{(2)}_{k_1} - (\overline{c}_{i'} - \underline{c}_{i'})x_{i'} + \pi^{(1)} \geq 0 \label{p7h}\\
& \pi^{(1)} \ge 0 \label{p7i}\\
& \rho^{(1)}_i \ge 0 & \forall i\in[n].
\end{align}
\end{subequations}
 Note that constraints \eqref{p7g} and \eqref{p7h} are consequences of the constraints $z^{(k_1)} \geq 0$ and $z^{(k_2)} \geq 0$ from problem \eqref{p5}.  Notice how $\pi^{(1)}$ cancels out in constraint~\eqref{p7c}. 

We claim that in an optimal solution, we have $x_{i'} = 1$. Indeed, assume that $x_{i'} = 0$.  Then, constraint \eqref{p7g} implies that $\pi^{(1)} \leq - \pi^{(2)}_{k_2}$. But note that $\pi^{(2)}_{k_2} \in \Pi$ and the set $\Pi$ contains only non-negative numbers by definition. Because we also have $\pi^{(1)} \geq 0$, this actually implies that $\pi^{(2)}_{k_2} = \pi^{(1)} = 0$. But inserting these values into the equality for $z^{(k_1)}$, we see that $z^{(k_1)} = 0$, which is a contradiction to the assumption that both $z^{(k_1)}$ and  $z^{(k_2)}$ are active, that is $z^{(k_1)}, z^{(k_2)} > 0$.
We conclude that $x_{i'}=1$. Furthermore, because the variable $\pi^{(1)}$ appears only in constraints \eqref{p7g}, \eqref{p7h} and \eqref{p7i}, we conclude that there is an optimal $\pi^{(1)}$ that is either equal to 
$\overline{c}_{i'} - \underline{c}_{i'} - \pi^{(2)}_{k_2}$, or equal to  
$\overline{c}_{i'} - \underline{c}_{i'} - \pi^{(2)}_{k_1}$, or equal to 0. Finally, we have
\[ \rho^{(1)}_i = [(\overline{c}_i - \underline{c}_i)x_i - \overline{c}_{i'} - \underline{c}_{i'}]_+ = [(\overline{c}_i - \underline{c}_i) - \overline{c}_{i'} - \underline{c}_{i'}]_+ x_i. \]
Incorporating all these simplifications into problem \eqref{p7}, one sees that problem~\eqref{p7} becomes one of type~\eqref{p8}, with continuous $\pmb{y}^{(1)},\pmb{y}^{(2)}$. 
 For fixed $\pmb{x}\in\{0,1\}^n$, the remaining problem decomposes into two independent problems in variables $\pmb{y}^{(1)},\pmb{y}^{(2)}$. Due to assumption~\eqref{avar}, we can thus consider $\pmb{y}^{(1)},\pmb{y}^{(2)}\in\{0,1\}^n$ instead of $\pmb{y}^{(1)},\pmb{y}^{(2)}\in [0,1]^n$. This means that we have transformed the problem into a problem of type \eqref{p8}.
 
In summary, in the second case of the main theorem, we have to solve one instance of the following problem of type \eqref{p8} for each tuple $(k_1, k_2, i', \pi^{(1)})$ with the property that $k_1 \in [K]$ and $k_2 \in [K], k_1 \neq k_2$ and $i' \in [n]$ and $\pi^{(1)} \in \{\overline{c}_{i'} - \underline{c}_{i'} - \pi^{(2)}_{k_2}, \overline{c}_{i'} - \underline{c}_{i'} - \pi^{(2)}_{k_1}, 0\}$ and $\pi^{(2)}_{k_1} - \overline{c}_{i'} + \underline{c}_{i'} + \pi^{(1)} \geq 0$ and $\overline{c}_{i'} - \underline{c}_{i'} - \pi^{(1)} - \pi^{(2)}_{k_2} \geq 0$ and $\pi^{(1)} \geq 0$. Note that there are $O(n^3)$ many tuples with these properties.
\begin{subequations}
\label{p10}
\begin{align}
\min\ & \tilde{\pmb{a}}^t\pmb{x} + \tilde{\pmb{b}}^t\pmb{y}^{(1)} + \tilde{\pmb{c}}^t\pmb{y}^{(2)} + \tilde{v} \\
\text{s.t. } 
& A(\pmb{x}+\pmb{y}^{(1)}) \ge \pmb{b}  \\
& A(\pmb{x}+\pmb{y}^{(2)}) \ge \pmb{b}  \\
& \pmb{x} + \pmb{y}^{(1)} \le \pmb{1} \\
& \pmb{x} + \pmb{y}^{(2)} \le \pmb{1} \\
& x_i \in\{0,1\} & \forall i\in[n] \\
& y^{(1)}_i,y^{(2)}_i \in \{0,1\} & \forall i\in[n]\\
\text{with constants }& \tilde{a}_i = \underline{c}_i + [(\overline{c}_i - \underline{c}_i) - \overline{c}_{i'} - \underline{c}_{i'}]_+ \nonumber \\
& \tilde{b}_i = (\underline{d}_i + [\overline{d}_i - \underline{d}_i - \pi^{(2)}_{k_1}]_+) \frac{(\overline{c}_{i'} - \underline{c}_{i'})x_{i'} - \pi^{(1)} - \pi^{(2)}_{k_2}}{\pi^{(2)}_{k_1} - \pi^{(2)}_{k_2}} \nonumber \\
& \tilde{c}_i = (\underline{d}_i + [\overline{d}_i - \underline{d}_i - \pi^{(2)}_{k_2}]_+) \frac{\pi^{(2)}_{k_1} - (\overline{c}_{i'} - \underline{c}_{i'})x_{i'} + \pi^{(1)}}{\pi^{(2)}_{k_1} - \pi^{(2)}_{k_2}} \nonumber \\
&\tilde{v} =  \Gamma (( \overline{c}_{i'} - \underline{c}_{i'})x_{i'} - \pi^{(1)} ) + \Gamma \pi^{(1)}. \nonumber 
\end{align}
\end{subequations}
We have thus completed the proof of Theorem~\ref{th:p}. We summarize our detailed findings in the following lemma, which implies the main theorem:
\begin{lemma}
Let a robust two-stage combinatorial optimization problem with two-stage continuous budgeted uncertainty be given. If assumption~\eqref{ass} holds, then a solution to the robust problem can be obtained by taking the minimum of the objective values of
\begin{itemize}
\item an instance of the nominal problem~\eqref{p6-dash} for each of the $O(n^2)$ values of $(k, \pi^{(1)})$, as described there, and
\item  an instance of the problem~\eqref{p10} (which is a problem of type~\eqref{p8}) for each of the $O(n^3)$ values of $(k_1, k_2, i', \pi^{(1)})$, as described there.
\end{itemize}
\end{lemma}
\begin{proof}
Due to Lemma~\ref{lemma:robust-reformulation}, problem~\eqref{p5} is an equivalent reformulation of the robust problem. Due to Lemma~\ref{lemmaz}, all solutions to problem~\eqref{p5} have the property that there is either exactly one value, or exactly two values of $k$ such that $z^{(k)} \neq 0$. If exactly one $z^{(k)}$ is non-zero, then the arguments above imply that the optimum solution can be found by taking the minimum of $O(n^2)$ instances of problem~\eqref{p6-dash}. If exactly two $z^{(k)}$ are non-zero, then the arguments above imply that the optimum solution can be found by taking the minimum of $O(n^3)$ instances of problem~\eqref{p10}.
\end{proof}

\subsection{A variant of continuous budgeted uncertainty}
\label{subsection:uvar}

We now consider a variant of the budgeted uncertainty sets with the difference that the budget bounds the absolute increase of the cost. That is, we consider sets
\[
    \cU^{\text{var}} = \left\{ \pmb{c}\in\mathbb{R}^n_+ : 
    c_i = \underline{c}_i + \delta_i,\ 
    \delta_i\in[0,\overline{c}_i - \underbar{c}_i]\ \forall i\in[n],\ 
    \sum_{i\in[n]} \delta_i \le \Gamma
     \right\}
\]
for the adversarial first stage and sets
\[
    \cU^{\text{var}}(\pmb{c}) = \left\{ \pmb{d}\in\mathbb{R}^n_+ : 
    d_i = \underline{d}_i + \delta_i,\  
    \delta_i\in[0,\overline{d}_i-\underline{d}_i]\ \forall i\in[n],\ 
    \sum_{i\in[n]} \delta_i \le \Gamma - \sum_{i\in[n]} (c_i - \underline{c}_i)
     \right\}
\]
for the adversarial second stage. Sets of this structure are sometimes used as well, see, e.g. \cite{nasrabadi2013robust,chassein2018recoverable}. In the following it does not make a difference if $\pmb{\delta}$ is continuous or discrete. In the latter case, we assume without loss of generality that $\overline{\pmb{c}}-\underline{\pmb{c}}$ and $\overline{\pmb{d}}-\underline{\pmb{d}}$ are integer. 

Similar to classic single stage min-max robust optimization problems with uncertainty set $\cU^{\text{var}}$, we show that it suffices to solve two instances of nominal optimization problems to determine an optimal solution for the two-stage problem.

\begin{theorem}\label{th:variant}
The robust two-stage combinatorial optimization problem with two-stage budgeted uncertainty variant $\cU^{\text{var}}$ can be decomposed into two problems of nominal type.
\end{theorem}
\begin{proof}
Let any first-stage solution $\pmb{x}\in\X$ be given. If the adversary decides to invest a budget $\gamma\le \Gamma$ on these items, the resulting costs are $\min\{ \underline{\pmb{c}}^t\pmb{x} + \gamma, \overline{\pmb{c}}^t\pmb{x}\}$. Hence, we can assume that $\gamma \le (\overline{\pmb{c}}-\underline{\pmb{c}})^t\pmb{x}$, which leads to costs $\pmb{c}^t\pmb{x} + \gamma$.
Now consider any second-stage solution $\pmb{y}\in\X(\pmb{x})$. The remaining budget $\Gamma-\gamma$ leads to second-stage costs $\min\{\underline{\pmb{d}}^t\pmb{y}+\Gamma-\gamma, \overline{\pmb{d}}^t\pmb{y}\}$. 
In total, the costs are thus
\begin{align*}
& \underline{\pmb{c}}^t\pmb{x} + \gamma + \min\{\underline{\pmb{d}}^t\pmb{y}+\Gamma-\gamma,\ \overline{\pmb{d}}^t\pmb{y}\} \\
= &\min\{ \underline{\pmb{c}}^t\pmb{x} + \gamma + \underline{\pmb{d}}^t\pmb{y}+\Gamma-\gamma,\  
\underline{\pmb{c}}^t\pmb{x} + \gamma + \overline{\pmb{d}}^t\pmb{y}
\} \\
= &\min\{ \underline{\pmb{c}}^t\pmb{x} + \underline{\pmb{d}}^t\pmb{y}+\Gamma,\ 
\underline{\pmb{c}}^t\pmb{x} + \gamma + \overline{\pmb{d}}^t\pmb{y}
\}
\end{align*}
Hence, for the adversary it is optimal to spend as much as possible in the first stage, using $\gamma = \min\{\Gamma,(\overline{\pmb{c}}-\underline{\pmb{c}})^t\pmb{x}\}$. Intuitively,  due to the fact that all cost increases to the items contribute equally to the used budget of the adversary, it does not make sense for the adversary to safe any usable budget to attack the first stage solution for the second stage. Thus, the robust problem is equivalent to
\[ \min\left\{ \Gamma + \min_{\pmb{x} \in \X} \left( \underline{\pmb{c}}^t \pmb{x} + \min_{\pmb{y} \in \X(\pmb{x})} \underline{\pmb{d}}^t \pmb{y}\right),\  \min_{\pmb{x} \in \X} \left( \bar{\pmb{c}}^t \pmb{x} + \min_{\pmb{y} \in \X(\pmb{x})} \bar{\pmb{d}}^t \pmb{y} \right) \right\}, \]
which can be rewritten as two nominal problems
\[ \min\left\{ \Gamma + \min_{\pmb{x} \in \X(\pmb{0})} \pmb{v}^t \pmb{x},  \min_{\pmb{x} \in \X(\pmb{0})} \pmb{w}^t \pmb{x} \right\} \]
with $v_i = \min\{\underline{c}_i, \underline{d}_i\}$ and $w_i = \min\{\bar{c}_i, \bar{d}_i\}$. Recall that $\X(\pmb{0})$ denotes the set of feasible solutions for the nominal problem.
\end{proof}
Note that Theorem~\ref{th:variant} does not require assumption~\eqref{ass}.

\section{Two-stage discrete budgeted uncertainty}
\label{sec:disc}

In this section we first derive compact mixed-integer programming models for the case of two-stage discrete budgeted uncertainty and then show that, contrary to the continuous case, a wide range of combinatorial optimization problems become NP-hard to solve. We finally consider a special case, where first-stage and second-stage costs coincide, and show that such problems become easier to solve. Like in the last section, we only consider optimization problems where assumption (\ref{ass}) holds.

\subsection{Models}

We now consider budgeted uncertainty sets as defined in Section~\ref{sec:cont} with the difference that the adversarial variables $\delta_i$ determining the distribution of the uncertainty budget need to be discrete, i.e., we have $\delta_i\in\{0,1\}$ instead of $\delta_i\in[0,1]$. While for one-stage problems, this does not have any impact on the problem, it is well known to make a difference for two-stage problems (see, e.g., \cite{chassein2018recoverable,goerigk2020recoverable}). 

Discrete variables in the inner adversarial recourse problem $\textsc{AdvRec}(\pmb{x},\pmb{c},\pmb{y})$ can be relaxed without changing the optimal objective value, which means that we find the same recourse problem $\textsc{Rec}(\pmb{x},\pmb{c})$ as before in~\eqref{p1}. We once again define the set  
\[\Pi := \{ \overline{d}_i - \underline{d}_i : i\in[n]\} \cup \{0\}\] 
of possible dual variable values, and let $\Pi = \{ \pi^{(2)}_1, \ldots, \pi^{(2)}_K\}$ to find the following adversarial problem $\textsc{Adv}(\pmb{x})$:
\begin{subequations}
\label{pd1}
\begin{align}
\max\ & t + \sum_{i\in[n]} (\underline{c}_i + (\overline{c}_i-\underline{c}_i)\delta_i)x_i  \\
\text{s.t. } & t \le 
(\pmb{b}-A\pmb{x})^t \pmb{\alpha}^{(k)} + (\pmb{x}-\pmb{1})^t\pmb{\beta}^{(k)}
+(\Gamma - \sum_{i\in[n]} \delta_i) \pi^{(2)}_k & \forall k\in[K] \\
&  (A^t\pmb{\alpha}^{(k)})_i - \beta^{(k)}_i \le \underline{d}_i + [\overline{d}_i - \underline{d}_i - \pi^{(2)}_k]_+ & \forall i\in[n], k \in [K]\\
& \sum_{i\in[n]} \delta_i \le \Gamma \\
& \delta_i \in \{0,1\} & \forall i\in[n] \\
& \pmb{\alpha}^{(k)} \in\mathbb{R}^m_+ & \forall k\in[K] \\
& \pmb{\beta}^{(k)} \in\mathbb{R}^n_+ &\forall k\in[K] 
\end{align}
\end{subequations}
The only difference to problem~\eqref{p3} lies in the variables $\pmb{\delta}$, which are now discrete instead of continuous. Let $[\Gamma]_0 := \{0,1,\ldots,\Gamma\}$. We guess the value of $\gamma:=\sum_{i\in[n]}\delta_i$, where we have that $\gamma \in [\Gamma]_0$. Problem~\eqref{pd1} hence decomposes to:
\begin{subequations}
\label{pd2}
\begin{align}
\max_{\gamma\in[\Gamma]_0} \max\ & t + \sum_{i\in[n]} (\underline{c}_i + (\overline{c}_i-\underline{c}_i)\delta_i)x_i  \\
\text{s.t. } & 
t \le (\pmb{b}-A\pmb{x})^t \pmb{\alpha}^{(k)} + (\pmb{x}-\pmb{1})^t\pmb{\beta}^{(k)}
+(\Gamma - \gamma) \pi^{(2)}_k & \forall k\in[K] \label{pd2con} \\
&  (A^t\pmb{\alpha}^{(k)})_i - \beta^{(k)}_i \le \underline{d}_i + [\overline{d}_i - \underline{d}_i - \pi^{(2)}_k]_+ & \forall i\in[n], k \in [K] \label{pd2c} \\
& \sum_{i\in[n]} \delta_i = \gamma \label{pd2d} \\
& \delta_i \in \{0,1\} & \forall i\in[n] \label{pd2e} \\
& \pmb{\alpha}^{(k)} \in\mathbb{R}^m_+ & \forall k\in[K] \\
& \pmb{\beta}^{(k)} \in\mathbb{R}^n_+ &\forall k\in[K]. 
\end{align}
\end{subequations}
Note that the constraints~\eqref{pd2con} do not depend on variables $\pmb{\delta}$ anymore. This means that variables $\pmb{\delta}$ and $\pmb{\alpha}^{(k)},\pmb{\beta}^{(k)}$ have become decoupled and we can relax problem~\eqref{pd2} for each choice of $\gamma$ without affecting the objective value by considering $\delta_i \in [0,1]$ instead of $\delta_i \in \{0,1\}$.  We dualize each subproblem, and introduce dual variables $z^{(k,\gamma)}, y^{(k,\gamma)}_i, \pi^{(\gamma)},$ and $\rho^{(\gamma)}_i$ for constraints \eqref{pd2con}, \eqref{pd2c}, \eqref{pd2d}, and \eqref{pd2e}, respectively. Combining the resulting problems, and integrating the variables and constraints of the first stage, we find the following non-linear compact formulation of the two-stage robust problem with two-stage discrete budgeted uncertainty $\textsc{Rob}$:
\begin{subequations}
\label{pd3}
\begin{align}
\min\ & \sum_{i\in[n]} \underline{c}_i x_i + t \\
\text{s.t. } & t \ge \sum_{k\in[K]} (\Gamma-\gamma) \pi^{(2)}_k z^{(k,\gamma)} \nonumber \\
& + \sum_{k\in[K]} \sum_{i\in[n]} (\underline{d}_i + [\overline{d}_i - \underline{d}_i - \pi^{(2)}_k]_+) y^{(k,\gamma)}_i + \gamma \pi^{(\gamma)} + \sum_{i\in[n]} \rho^{(\gamma)}_i \hspace{-1cm} & \forall \gamma \in [\Gamma]_0 \\
& \pi^{(\gamma)} + \rho^{(\gamma)}_i \ge (\overline{c}_i - \underline{c}_i) x_i & \forall i\in[n], \gamma\in [\Gamma]_0 \\
& A \pmb{y}^{(k,\gamma)} \ge (\pmb{b} - A\pmb{x}) z^{(k,\gamma)} & \forall k\in[K],\gamma \in [\Gamma]_0 \\
& (1-x_i) z^{(k,\gamma)} \ge y^{(k,\gamma)}_i & \forall i\in[n], k\in[K], \gamma \in [\Gamma]_0 \\
& \sum_{k\in[K]} z^{(k,\gamma)} = 1 & \forall \gamma\in[\Gamma]_0 \\
& x_i \in\{0,1\} & \forall i\in[n] \\
& y^{(k,\gamma)}_i \ge 0 & \forall i\in[n], k\in[K], \gamma\in[\Gamma]_0 \\
& z^{(k,\gamma)} \ge 0 & \forall k\in[K], \gamma\in[\Gamma]_0  \\
& \pi^{(\gamma)} \ge 0 & \forall \gamma\in[\Gamma]_0 \\
& \rho^{(\gamma_i)} \ge 0 & \forall \gamma \in[\Gamma]_0, i\in[n].
\end{align}
\end{subequations}
Note that the products between $x_i$ and $z^{(k,\gamma)}$ variables can be linearized, which results in a mixed-integer programming formulation. We can analyze variables $z^{(k,\gamma)}$ in a similar way as in Section~\ref{sec:cont}, as the following lemma states.
\begin{lemma}\label{lem:disc}
There is an optimal solution to problem~\eqref{pd3} where for each $\gamma\in [\Gamma]_0$, there is one $k'\in[K]$ such that
\[
z^{(k,\gamma)} = \begin{cases} 1 & \text{ if } k=k', \\
0 & \text{ otherwise.} 
\end{cases}
\]
\end{lemma}
\begin{proof}
Let us assume that in problem~\eqref{pd3}, variables $\pmb{x}$, $\pmb{\pi}$ and $\pmb{\rho}$ are fixed. The remaining problem then reduces to the following linear program:
\begin{subequations}
\label{pd4}
\begin{align}
\min\ &  t \\
\text{s.t. } & t \ge \sum_{k\in[K]} (\Gamma-\gamma) \pi^{(2)}_k z^{(k,\gamma)} \nonumber \\
& + \sum_{k\in[K]} \sum_{i\in[n]} (\underline{d}_i + [\overline{d}_i - \underline{d}_i - \pi^{(2)}_k]_+) y^{(k,\gamma)}_i + \text{const} & \forall \gamma \in [\Gamma]_0 \label{pd4con}\\
& A \pmb{y}^{(k,\gamma)} \ge (\pmb{b} - A\pmb{x}) z^{(k,\gamma)} & \forall k\in[K],\gamma \in [\Gamma]_0 \\
& (1-x_i) z^{(k,\gamma)} \ge y^{(k,\gamma)}_i & \forall i\in[n], k\in[K], \gamma \in [\Gamma]_0 \\
& \sum_{k\in[K]} z^{(k,\gamma)} = 1 & \forall \gamma\in[\Gamma]_0 \\
& y^{(k,\gamma)}_i \ge 0 & \forall i\in[n], k\in[K], \gamma\in[\Gamma]_0 \\
& z^{(k,\gamma)} \ge 0 & \forall k\in[K], \gamma\in[\Gamma]_0. 
\end{align}
\end{subequations}
Problem~\eqref{pd4} can be decomposed by minimizing a subproblem for each $\gamma$ with the right-hand side of constraint~\eqref{pd4con} in the objective, that is,
\begin{subequations}
\label{pd5}
\begin{align}
\min\ & \sum_{k\in[K]} (\Gamma-\gamma) \pi^{(2)}_k z^{(k,\gamma)}  + \sum_{k\in[K]} \sum_{i\in[n]} (\underline{d}_i + [\overline{d}_i - \underline{d}_i - \pi^{(2)}_k]_+) y^{(k,\gamma)}_i + \text{const} \hspace{-2cm}\\
& A \pmb{y}^{(k,\gamma)} \ge (\pmb{b} - A\pmb{x}) z^{(k,\gamma)} & \forall k\in[K] \\
& (1-x_i) z^{(k,\gamma)} \ge y^{(k,\gamma)}_i & \forall i\in[n], k\in[K] \\
& \sum_{k\in[K]} z^{(k,\gamma)} = 1 \\
& y^{(k,\gamma)}_i \ge 0 & \forall i\in[n], k\in[K] \\
& z^{(k,\gamma)} \ge 0 & \forall k\in[K].
\end{align}
\end{subequations}
We substitute $y^{(k,\gamma)}_i = z^{(k,\gamma)}\tilde{y}^{(k,\gamma)}_i$ analogous to the substitution of variables in problem~\eqref{p5} to find:
\begin{subequations}
\begin{align}
\min\ & \sum_{k\in[K]} (\Gamma-\gamma) \pi^{(2)}_k z^{(k,\gamma)}  + \sum_{k\in[K]} \sum_{i\in[n]} (\underline{d}_i + [\overline{d}_i - \underline{d}_i - \pi^{(2)}_k]_+) z^{(k,\gamma)}\tilde{y}^{(k,\gamma)}_i + \text{const} \hspace{-2cm}\\
& A \tilde{\pmb{y}}^{(k,\gamma)} z^{(k,\gamma)} \ge (\pmb{b} - A\pmb{x}) z^{(k,\gamma)} & \forall k\in[K] \label{pd5bcon1}\\
& (1-x_i) z^{(k,\gamma)} \ge z^{(k,\gamma)}\tilde{y}^{(k,\gamma)}_i & \forall i\in[n], k\in[K] \label{pd5bcon2}\\
& \sum_{k\in[K]} z^{(k,\gamma)} = 1 \\
& \tilde{y}^{(k,\gamma)}_i \ge 0 & \forall i\in[n], k\in[K] \\
& z^{(k,\gamma)} \ge 0 & \forall k\in[K].
\end{align}
\end{subequations}
We see that if $z^{(k,\gamma)} = 0$, then variables $\tilde{\pmb{y}}^{(k,\gamma)}$ can be chosen arbitrarily. Consider a feasible solution where all variables $\tilde{\pmb{y}}^{(k,\gamma)}$ are fixed. Then constraints~\eqref{pd5bcon1} and ~\eqref{pd5bcon2} are true independent of $z^{(k,\gamma)}$. The remaining problem in $z^{(k,\gamma)}$ is a linear minimization problem over a simplex. Hence the claim follows.
\end{proof}
Using Lemma~\ref{lem:disc}, we can assume that variables $z^{(k,\gamma)}$ in 
problem~\eqref{pd3} are binary. Note, however, that we cannot enumerate all possible choices of $z^{(k,\gamma)}$ in polynomial time, if $\Gamma$ is not a constant value. For every $\gamma \in [\Gamma]_0$, we need to decide independently for which $k \in [K]$ we have $z^{(k,\gamma)} = 1$. These are $O(K^{\Gamma+1})$ many possibilities. We write $\kappa^{(2)}_\gamma$ for the value of $\pi^{(2)}_k$ where $z^{(k,\gamma)}$ is equal to one.

Finally, consider variables $\pi^{(\gamma)}$ and $\rho^{(\gamma)}_i$. Using the classical argument of finding the kink points in a piecewise linear function \cite{bertsimas2003robust}, we find that $\pi^{(\gamma)}\in\{\overline{c}_i-\underline{c}_i : i\in[n]\} \cup \{0\}$. Combining these observations, we can enumerate the choice of all $z^{(k,\gamma)}$ and $\pi^{(\gamma)}$ variables in polynomial time if $\Gamma$ is a constant.

\begin{theorem}
Let a robust two-stage combinatorial optimization problem with two-stage discrete budgeted uncertainty be given. If assumption~\eqref{ass} holds, then the robust problem can be decomposed into $O(n^{2\Gamma})$ many subproblems of the form:
\begin{subequations}
\label{pd6}
\begin{align}
\min\ & t \\
\text{s.t. } & t \ge (\Gamma-\gamma)\kappa^{(2)}_\gamma + \sum_{i\in[n]} (\underline{d}_i + [\overline{d}_i - \underline{d}_i - \kappa^{(2)}_\gamma]_+)y^{(\gamma)}_i \nonumber \\
& + \gamma\pi^{(\gamma)} + \sum_{i\in[n]} (\underline{c}_i + [\overline{c}_i - \underline{c}_i - \pi^{(\gamma)}]_+)x_i & \forall \gamma\in [\Gamma]_0 \\
& A(\pmb{x}+\pmb{y}^{(\gamma)}) \ge \pmb{b}  & \forall \gamma\in [\Gamma]_0 \\
& \pmb{x} + \pmb{y}^{(\gamma)} \le \pmb{1} & \forall \gamma\in [\Gamma]_0 \\
& x_i \in\{0,1\} & \forall i\in[n] \\
& y^{(\gamma)}_i \in [0,1] & \forall  \gamma\in [\Gamma]_0, i\in[n].
\end{align}
\end{subequations}

\end{theorem}

Finally, note that we can use problems~\eqref{pd6} to find an alternative compact problem formulation. To this end, we consider $\pi^{(\gamma)}$ and $\kappa^{(2)}_\gamma$ as variables again. The brackets that enforce the positive part are replaced by variables $\rho^{(2)}_{\gamma,i}$ and $\rho^{(\gamma)}_i$, respectively. Note, however, that we must require $y^{(\gamma)}_i\in\{0,1\}$ for this replacement. This gives the following result.

\begin{theorem}
The robust two-stage combinatorial optimization problem with two-stage discrete budgeted uncertainty
under assumption~\eqref{ass} can be formulated as the following compact mixed-integer program:
\begin{subequations}
\label{pd7}
\begin{align}
\min\ & t \\
\text{s.t. } & t \ge (\Gamma-\gamma)\kappa^{(2)}_\gamma + \sum_{i\in[n]} \rho^{(2)}_{\gamma,i} + \sum_{i\in[n]} \underline{d}_i y^{(\gamma)}_i  \nonumber \\
& + \gamma\pi^{(\gamma)} + \sum_{i\in[n]} \rho^{(\gamma)}_i + \sum_{i\in[n]} \underline{c}_i x_i & \forall \gamma\in [\Gamma]_0  \label{16b}\\
& A(\pmb{x}+\pmb{y}^{(\gamma)}) \ge \pmb{b}  & \forall \gamma\in [\Gamma]_0 \\
& \pmb{x} + \pmb{y}^{(\gamma)} \le \pmb{1} & \forall \gamma\in [\Gamma]_0 \\
& \kappa^{(2)}_\gamma + \rho^{(2)}_{\gamma,i} \ge (\overline{d}_i - \underline{d}_i)y^{(\gamma)}_i & \forall \gamma\in [\Gamma]_0, i\in[n] \label{16e}\\
& \pi^{(\gamma)} + \rho^{(\gamma)}_i \ge (\overline{c}_i - \underline{c}_i)x_i & \forall \gamma\in [\Gamma]_0, i\in[n] \label{16f} \\
& x_i \in\{0,1\} & \forall i\in[n] \\
& y^{(\gamma)}_i \in \{0,1\} & \forall  \gamma\in [\Gamma]_0, i\in[n] \\
& \pi^{(\gamma)}, \kappa^{(2)}_\gamma \ge 0 & \forall  \gamma\in [\Gamma]_0 \\
& \rho^{(\gamma)}_i, \rho^{(2)}_{\gamma,i} \ge 0 & \forall  \gamma\in [\Gamma]_0, i\in[n] .
\end{align}
\end{subequations}

\end{theorem}

\subsection{Hardness}
\label{sec:hardness}

While the results in Section~\ref{sec:cont} show that problems with two-stage continuous budgeted uncertainty can often be solved in polynomial time, we show here that this is unlikely to be possible for two-stage discrete budgeted uncertainty, as simple problems already become NP-hard. Recall that for \textsc{representative selection}, we have $\X = \{ \pmb{x}\in\{0,1\}^n : \sum_{i\in T_j} x_i=1, j\in[\ell]\}$ with an item partition $T_1\cup T_2 \cup \ldots \cup T_\ell = [n]$.

\begin{theorem}\label{th:hardness}
Two-stage \textsc{representative selection} with two-stage discrete budgeted uncertainty is NP-hard, even if $|T_j|=2$ and $\Gamma=1$.
\end{theorem}
\begin{proof}
Let an instance $a_1, \dots, a_n \in \mathbb{N}$ of the NP-hard \textsc{partition} problem be given, see \cite{garey1979computers}. Let $\sum_{i\in[n]} a_i = 2A$. The question is whether there exists a set $X\subseteq[n]$ such that $\sum_{i\in X} a_i = \sum_{i\in \bar{X}} a_i = A$ with $\bar{X} = [n]\setminus X$.

We construct an instance of the two-stage \textsc{representative selection} problem in the following way. There are $n$ sets of items $T_i$, each consisting of two items from which one must be chosen. For the first item in each set, we set $\underline{c}_{i1} = -a_i$, $\hat{c}_{i1} = 4A$, $\underline{d}_{i1} = a_i$ and $\hat{d}_{i1} = 0$. For the second item, we set $\underline{c}_{i2} = -a_i$, $\hat{c}_{i2} = 4A$, $\underline{d}_{i2} = -3a_i$, $\hat{d}_{i2} = M$ for a big constant $M$ (it suffices to set $M>2A$). Refer to Table~\ref{tab:red1} for an overview. The instance is completed by setting $\Gamma=1$.

\begin{table}[htb]
\begin{center}
\begin{tabular}{c|cc|cc|c|cc}
\multicolumn{1}{c}{} & \multicolumn{2}{c}{$T_1$} & \multicolumn{2}{c}{$T_2$} & \multicolumn{1}{c}{} & \multicolumn{2}{c}{$T_n$} \\
$i$ & 1 & 2 & 3 & 4 & $\cdots$ & $2n-1$ & $2n$  \\
\hline
$\underline{c}_i$ & $-a_1$ & $-a_1$ & $-a_2$ & $-a_2$ & $\cdots$ & $-a_n$ & $-a_n$\\ 
$\overline{c}_i-\underline{c}_i$ &  $4A$ & $4A$ & $4A$ & $4A$ & $\cdots$ & $4A$ & $4A$\\ 
$\underline{d}_i$ & $a_1$ & $-3a_1$ & $a_2$ & $-3a_2$ & $\cdots$ & $-a_n$ & $-3a_n$ \\
$\overline{d}_i-\underline{d}_i$ & $0$ & $M$ & $0$ & $M$ & $\cdots$ & $0$ & $M$
\end{tabular}
\caption{Instance used in the hardness reduction for two-stage \textsc{representative selection}.}\label{tab:red1}
\end{center}
\end{table}

Note that we can build an instance without negative item costs in the same manner by adding a sufficiently large constant to each cost. As every feasible solution contains the same number of items, the objective value of each solution is changed by the same constant.

Furthermore, note that the first-stage costs of items 1 and 2 in each set are the same. This means that we only need to consider the following choices in each set: (i) Buy any of the two items in the first stage. (ii) Buy the first item in the second stage. (iii) Buy the second item in the second stage. Let $X\subseteq[n]$ be the index set of those sets $T_i$ where we decide to buy any of the two items in the first stage.

The adversary now has two choices: Either the uncertainty budget is spent on the first-stage choice of items, or it is spent on the future second-stage choice of items. In the first case, a cost increase of $4A$ is reached, as long as any item is bought in the first stage. Then, an optimal solution will pack the second item in each remaining set, as each has lower costs than the first respective item. Overall, the costs become $\sum_{i\in X} -a_i + 4A -\sum_{i\in\bar{X}} -3a_i$. In the second case, packing any of the second items in each set will lead to a high penalty $M$ and is thus not possible in an optimal solution. Hence, we need to fill the solution with the first item from each set. The complete costs become $\sum_{i\in X} -a_i + \sum_{i\in \bar{X}} a_i$.

Combining these two cases, the worst-case costs of any choice $X$ is:
\begin{align*}
&\max\{ -\sum_{i\in X} a_i + 4A - 3\sum_{i\in \bar{X}} a_i , -\sum_{i\in X} a_i + \sum_{i \in \bar{X}} a_i \} \\
=& \max\{ 2\sum_{i\in X} a_i + 2 \sum_{i\in \bar{X}} a_i - \sum_{i\in X} a_i - 3\sum_{i\in \bar{X}} a_i, -\sum_{i\in X} a_i + \sum_{i \in \bar{X}} a_i \} \\
=& \max \{\sum_{i\in X} a_i - \sum_{i\in \bar{X}} a_i, -\sum_{i\in X} a_i + \sum_{i\in \bar{X}} a_i \} \\
=& \left| \sum_{i\in X} a_i - \sum_{i\in \bar{X}} a_i \right|.
\end{align*}
Let $X$ be an optimal choice for this problem. We see that the \textsc{partition} problem is a yes-instance if and only if the objective value of $X$ in the two-stage \textsc{representative selection} problem is zero.
\end{proof}

As the \textsc{representative selection} problem can be interpreted as a graph-based connectivity problem, where every set $T_j$ corresponds to a set of parallel edges, which are then arranged sequentially, we can conclude that also \textsc{shortest path} and \textsc{minimum spanning tree} problems are hard.
  Furthermore, the \textsc{representative selection} problem can be reduced to an instance of the \textsc{Assignment} problem, which hence also is hard.
We add four vertices for each set $T_j$ (two on each side of the bipartite graph) and set the costs of the two 
only possible assignments among those four vertices such that it corresponds to the two possible selections in $T_j$. 
All other costs are set to large values.

\begin{cor}
Two-stage \textsc{shortest path}, \textsc{minimum spanning tree} and \textsc{Assignment} problems with two-stage discrete budgeted uncertainty are NP-hard, even for series-parallel graphs and if $\Gamma=1$.
\end{cor}

We now show that hardness also holds for the \textsc{selection} problem, where $\X = \{ \pmb{x}\in\{0,1\}^n : \sum_{i\in[n]} x_i= p\}$.

\begin{theorem}
Two-stage \textsc{selection} with two-stage discrete budgeted uncertainty is NP-hard, even if $\Gamma=1$.
\end{theorem}
\begin{proof}
As before, let an instance $a_1, \dots, a_n \in \mathbb{N},\sum_{i\in[n]} a_i = 2A$ of the NP-hard \textsc{partition} problem be given, see \cite{garey1979computers}. The questions is whether there exists a set $X\subseteq[n]$ such that $\sum_{i\in X} a_i = \sum_{i\in \bar{X}} a_i = A$ with $\bar{X} = [n]\setminus X$.

We construct an instance of the two-stage \textsc{selection} problem in the following way. There are $2n+1$ items, where $n$ items are of type ($\alpha$), one item is of type ($\beta$), and $n$ items are of type ($\gamma$). Items of type ($\alpha$) each correspond to one of the given weights $a_i$, with $\underline{c}_i=\overline{c}_i=a_i$ and $\underline{d}_i = \overline{d}_i = 2a_i$. The item of type ($\beta$) has $\underline{c}_i = 0$, $\overline{c}_i = 2A$, and $\underline{d}_i=\overline{d}_i=M$, where $M\ge 4A$ denotes denotes a sufficiently large constant such that items with this cost will never be packed in an optimal solution. Finally, items of type ($\gamma$) all have $\underline{c}_i = \overline{c}_i = \overline{d}_i = M$ and $\underline{d}_i=0$, see Table~\ref{tab:red2}. We complete the instance by setting $p=n+1$ and $\Gamma=1$.

\begin{table}[htb]
\begin{center}
\begin{tabular}{c|ccc|c|ccc}
\multicolumn{1}{c}{} & \multicolumn{3}{c}{$\alpha$} & \multicolumn{1}{c}{$\beta$} & \multicolumn{3}{c}{$\gamma$} \\[-1ex]
\multicolumn{1}{c}{} & \multicolumn{3}{c}{\downbracefill} & \multicolumn{1}{c}{\downbracefill} & \multicolumn{3}{c}{\downbracefill} \\[2ex]
$i$ & 1 & $\dots$ & $n$ & $n+1$ & $n+2$ & $\dots$ & $2n+1$ \\
\hline
$\underline{c}_i$ & $a_1$ & $\dots$ & $a_n$ & 0 & $M$ & $\dots$ & $M$ \\ 
$\overline{c}_i$ & $a_1$ & $\dots$ & $a_n$ & $2A$ & $M$ & $\dots$ & $M$ \\ 
$\underline{d}_i$ & $2a_1$ & $\dots$ & $2a_n$ & $M$ & 0 & $\dots$ & 0 \\
$\overline{d}_i$ & $2a_1$ & $\dots$ & $2a_n$ & $M$ & $M$ & $\dots$ & $M$ 
\end{tabular}
\caption{Instance used in the hardness reduction for two-stage \textsc{selection}.}\label{tab:red2}
\end{center}
\end{table}

We may assume that no first-stage solution packs any item of type ($\gamma$). Let us assume that the item of type ($\beta$) is not packed. Then the adversary will not increase any item costs in the first stage, but is able to give one of the items packed in the second stage the high cost $M$. Hence, any optimal solution will pack item $\beta$ already in the first stage. Let $X\subseteq[n]$ be the index set of items of type ($\alpha$) packed in the first stage. The adversary now has two possible strategies: Either to increase the costs of item $\beta$, or to save the budget for the second stage. In the first case we can pack items of type ($\gamma$) to reach a complete solution with no additional costs. In the second case we are forced to complete our solution by packing items of type ($\alpha$). The total costs are hence
\[ \sum_{i\in X} a_i + \max \left\{ 2A, \sum_{i\in \bar{X}} 2a_i \right\} =  2A  + \max\left\{ \sum_{i\in X} a_i,  \sum_{i\in \bar{X}} a_i \right\}. \]
We find that the \textsc{partition} problem is a yes-instance if and only if we can choose a solution $X$ to the two-stage \textsc{selection} problem with total costs less or equal to $3A$.

\end{proof}

\subsection{Special case of equal costs}

The results in Section~\ref{sec:hardness} show that most combinatorial optimization problems become 
hard to solve under two-stage discrete budgeted uncertainty. We now show that these results do not necessarily hold for the specific case that $\underline{\pmb{c}}=\underline{\pmb{d}}$ and $\overline{\pmb{c}} = \overline{\pmb{d}}$. That is, we consider the following problem: First the decision maker chooses a set of items and pays $\underline{c}_i$ for each such item. Then the adversary can choose up to $\Gamma$ many of these items and increase their costs to $\overline{c}_i$, thereby forcing the decision maker to pay the additional cost difference $\overline{c}_i - \underline{c}_i$ for all items whose costs were increased. This process is now repeated, where the decision maker chooses a second set of items with costs $\underline{c}_i$ to create a feasible solution, and the adversary can spend the remaining budget to increase item costs. This setting is hence similar to the classic setting of one-stage robust optimization insofar there are no different cost vectors between the first and the second stage, see Table~\ref{tab:ex} in Section~\ref{sec:intro} for an example of this setting.

Consider the case $\Gamma=1$ in problem~\eqref{pd7}. 
Note that $\kappa^{(2)}_1$ and $\pi^{(0)}$ are multiplied with zero in constraint~\eqref{16b}. We can therefore make them sufficiently large to fulfill constraints~\eqref{16e} for $\gamma=1$ and \eqref{16f} for $\gamma=0$, respectively. Hence, after
appropriately renaming variables, the problem formulation is equivalent to the following mixed-integer program:
\begin{align*}
\min\ & t \\
\text{s.t. } & t \ge \pi^{(0)} + \sum_{i\in[n]} \rho^{(0)}_i + \sum_{i\in[n]} \underline{c}_i y^{(0)}_i + \sum_{i\in[n]} \underline{c}_i x_i \\
 & t \ge \pi^{(1)} + \sum_{i\in[n]} \rho^{(1)}_i + \sum_{i\in[n]} \underline{c}_i y^{(1)}_i + \sum_{i\in[n]} \underline{c}_i x_i \\
& A(\pmb{x}+\pmb{y}^{(0)}) \ge \pmb{b}   \\
& A(\pmb{x}+\pmb{y}^{(1)}) \ge \pmb{b}   \\
& \pmb{x} + \pmb{y}^{(0)} \le \pmb{1} \\
& \pmb{x} + \pmb{y}^{(1)} \le \pmb{1} \\
& \pi^{(0)}+ \rho^{(0)}_i \ge (\overline{c}_i - \underline{c}_i)y^{(0)}_i & \forall  i\in[n] \\
& \pi^{(1)}+ \rho^{(1)}_i \ge (\overline{c}_i - \underline{c}_i)x_i & \forall  i\in[n] \\
& x_i \in\{0,1\} & \forall i\in[n] \\
& y^{(0)}_i,y^{(1)}_i \in \{0,1\} & \forall  i\in[n] \\
& \pi^{(0)}, \pi^{(1)} \ge 0  \\
& \rho^{(0)}_i, \rho^{(1)}_i \ge 0 & \forall  i\in[n]. 
\end{align*}
Note that the variables $\pi^{(0)}$ and $\rho^{(0)}_1, \dots, \rho^{(0)}_n$ appear only in two constraints (apart from the non-negativity constraints). An analogous observation holds for the variables $\pi^{(0)}$ and $\rho^{(0)}_1, \dots, \rho^{(0)}_n$. This fact together with the structure of those constraints implies that the problem has an optimal solution where $\pmb{\rho}^{(0)}=\pmb{\rho}^{(1)}=\pmb{0}$ and 
\begin{align*}
 \pi^{(0)} &=\max\{ \overline{c}_i - \underline{c}_i : i\in[n], y^{(0)}_i = 1\}\text{ and} \\
\pi^{(1)} &= \max\{ \overline{c}_i -\underline{c}_i : i\in[n], x_i =1\}.
\end{align*} 
Note that if $\rho^{(0)}_i > 0$, then one can decrease $\rho^{(0)}_i > 0$ and increase $\pi^ {(0)}$ by the same amount. This means that in an optimal solution, both $\pi^ {(0)}$ and $\pi^ {(1)}$ take one of the $O(n)$ values in the set $\{ \overline{c}_i - \underline{c}_i : i\in[n] \}$. We can therefore enumerate possible values for $\pi^{(0)}$ and $\pi^{(1)}$. Each subproblem is then as follows:
\begin{subequations}
\label{special2}
\begin{align}
\min\ & t \\
\text{s.t. } & t \ge \pi^{(0)} + \sum_{i\in[n]} \underline{c}_i y^{(0)}_i + \sum_{i\in[n]} \underline{c}_i x_i \\
 & t \ge \pi^{(1)} + \sum_{i\in[n]} \underline{c}_i y^{(1)}_i + \sum_{i\in[n]} \underline{c}_i x_i \\
& A(\pmb{x}+\pmb{y}^{(0)}) \ge \pmb{b}   \\
& A(\pmb{x}+\pmb{y}^{(1)}) \ge \pmb{b}   \\
& \pmb{x} + \pmb{y}^{(0)} \le \pmb{1} \\
& \pmb{x} + \pmb{y}^{(1)} \le \pmb{1} \\
& y^{(0)}_i = 0 & \forall i\in[n]: \overline{c}_i - \underline{c}_i > \pi^{(0)} \label{forbid1}\\
& x_i = 0 & \forall i\in[n]: \overline{c}_i - \underline{c}_i > \pi^{(1)} \label{forbid2}\\
& x_i \in\{0,1\} & \forall i\in[n] \\
& y^{(0)}_i,y^{(1)}_i \in \{0,1\} & \forall  i\in[n]. 
\end{align}
\end{subequations}
By constraints~\eqref{forbid1} and \eqref{forbid2}, some variables are forced to be zero. In the following, we refer to the corresponding variables as being ''forbidden''.

\begin{theorem}\label{th:special}
The two-stage \textsc{representative selection} problem with two-stage discrete budgeted uncertainty and with $|T_j|=2$, $\Gamma=1$, $\underline{\pmb{c}}=\underline{\pmb{d}}$ and $\overline{\pmb{c}} = \overline{\pmb{d}}$ can be solved in $O(n^3)$ time.
\end{theorem}

\begin{proof}
We show that problem~\eqref{special2} can be solved in $O(n)$ time for the \textsc{representative selection} problem with $|T_j|=2$, implying the claimed result. 
The objective of problem~\eqref{special2} is to minimize $\sum_{i\in[n]} \underline{c}_i x_i + \max\{\pi^{(0)} + \sum_{i\in[n]} \underline{c}_i y^{(0)}_i,  \pi^{(1)} + \sum_{i\in[n]} \underline{c}_i y^{(1)}_i\}$, while constraints~\eqref{forbid1} and \eqref{forbid2} imply that some choices in $\pmb{x}$ and $\pmb{y}^{(0)}$ are forbidden.

Let us denote as items 1 and 2 an arbitrary part where one of the two items need to be chosen. Let $\underline{c}_1 \le \underline{c}_2$. 
If it is possible to choose $x_1$, then this is not worse than any other choice. If $x_1$ is forbidden, then we take the cheapest option allowed between $y^{(0)}_1$ and $y^{(0)}_2$ in combination with $y^{(1)}_1$. Only if both choices in $y^{(0)}$ are forbidden, the only feasible choice is to pack $x_2$. If also $x_2$ is not allowed, then the problem is infeasible. Table~\ref{tab:special} lists all possible cases how item choices may be forbidden (marked with X) with an optimal solution for each case. To solve the problem, we hence only need to iterate through all items once.

\begin{table}[htb]
\begin{center}
\begin{tabular}{cccc}
\begin{tabular}{r|rrr}
$i$ & $x_i$ & $y^{(0)}_i$ & $y^{(1)}_i$ \\
 \hline
1 & 1 & 0 & 0 \\
2 & 0 & 0 & 0 
\end{tabular}
&
\begin{tabular}{r|rrr}
$i$ & $x_i$ & $y^{(0)}_i$ & $y^{(1)}_i$ \\
  \hline
1 & X & 1 & 1 \\
2 & 0 & 0 & 0 
\end{tabular}
&
\begin{tabular}{r|rrr}
$i$ & $x_i$ & $y^{(0)}_i$ & $y^{(1)}_i$ \\
  \hline
1 & 1 & 0 & 0 \\
2 & X & 0 & 0 
\end{tabular}
&
\begin{tabular}{r|rrr}
$i$ & $x_i$ & $y^{(0)}_i$ & $y^{(1)}_i$ \\
  \hline
1 & X & 1 & 1 \\
2 & X & 0 & 0 
\end{tabular}
\\[1cm]
\begin{tabular}{r|rrr}
$i$ & $x_i$ & $y^{(0)}_i$ & $y^{(1)}_i$ \\
  \hline
1 & 1 & X & 0 \\
2 & 0 & 0 & 0 
\end{tabular}
&
\begin{tabular}{r|rrr}
$i$ & $x_i$ & $y^{(0)}_i$ & $y^{(1)}_i$ \\
  \hline
1 & 1 & 0 & 0 \\
2 & 0 & X & 0 
\end{tabular}
&
\begin{tabular}{r|rrr}
$i$ & $x_i$ & $y^{(0)}_i$ & $y^{(1)}_i$ \\
  \hline
1 & 1 & X & 0 \\
2 & 0 & X & 0 
\end{tabular}
&
\begin{tabular}{r|rrr}
$i$ & $x_i$ & $y^{(0)}_i$ & $y^{(1)}_i$ \\
  \hline
1 & X & X & 1 \\
2 & 0 & 1 & 0 
\end{tabular}
\\[1cm]
\begin{tabular}{r|rrr}
$i$ & $x_i$ & $y^{(0)}_i$ & $y^{(1)}_i$ \\
  \hline
1 & X & X & 0 \\
2 & 1 & X & 0 
\end{tabular}
&
\begin{tabular}{r|rrr}
$i$ & $x_i$ & $y^{(0)}_i$ & $y^{(1)}_i$ \\
  \hline
1 & X & X & 1 \\
2 & X & 1 & 0 
\end{tabular}
&
\begin{tabular}{r|rrr}
$i$ & $x_i$ & $y^{(0)}_i$ & $y^{(1)}_i$ \\
  \hline
1 & 1 & 0 & 0 \\
2 & X & X & 0 
\end{tabular}
&
\begin{tabular}{r|rrr}
$i$ & $x_i$ & $y^{(0)}_i$ & $y^{(1)}_i$ \\
  \hline
1 & X & 1 & 1 \\
2 & X & X & 0 
\end{tabular}
\\[1cm]
\begin{tabular}{r|rrr}
$i$ & $x_i$ & $y^{(0)}_i$ & $y^{(1)}_i$ \\
  \hline
1 & 1 & X & 0 \\
2 & X & X & 0 
\end{tabular}
&
\begin{tabular}{r|rrr}
$i$ & $x_i$ & $y^{(0)}_i$ & $y^{(1)}_i$ \\
  \hline
1 & X & X & 0 \\
2 & X & X & 0 
\end{tabular}
&
&
\end{tabular}
\end{center}
\caption{Proof of Theorem~\ref{th:special}: Possible item choices.}\label{tab:special}
\end{table}

\end{proof}

We show that it is also possible to show a similar result for \textsc{selection}.

\begin{theorem}\label{th:special2}
The two-stage \textsc{selection} problem with two-stage discrete budgeted uncertainty and with $\Gamma=1$, $\underline{\pmb{c}}=\underline{\pmb{d}}$ and $\overline{\pmb{c}} = \overline{\pmb{d}}$ can be solved in $O(n^5)$ time.
\end{theorem}
\begin{proof}
We show that problem~\eqref{special2} can be solved in $O(n^3)$ time for the \textsc{selection} problem, implying the claimed result. 

We sort all items such that $\underline{c}_1 \le \underline{c}_2 \le \ldots \le\underline{c}_n$. We define different types of items:
\begin{enumerate}
\item Items $\mathcal{A}= \{i\in[n] : \overline{c}_i-\underline{c}_i > \pi^{(0)},\ \overline{c}_i-\underline{c}_i > \pi^{(1)}\}$, i.e., items that are forbidden for both $\pmb{y}^{(0)}$ and $\pmb{x}$. 

\item Items $\mathcal{B} = \{i\in[n] : \overline{c}_i-\underline{c}_i \le \pi^{(0)},\ \overline{c}_i-\underline{c}_i \le \pi^{(1)}\}$, i.e., items that can be packed freely.

\item Items $\mathcal{C} = \{i\in[n] : \overline{c}_i-\underline{c}_i > \pi^{(0)},\ \overline{c}_i-\underline{c}_i \le \pi^{(1)}\}$, i.e., items that are only forbidden for $\pmb{y}^{(0)}$. Such items only exist if $\pi^{(0)} < \pi^{(1)}$.

\item Items $\mathcal{D} = \{i\in[n] : \overline{c}_i-\underline{c}_i \le \pi^{(0)},\ \overline{c}_i-\underline{c}_i > \pi^{(1)}\}$, i.e., items that are only forbidden for $\pmb{x}$. Such items only exist if $\pi^{(0)} > \pi^{(1)}$.
\end{enumerate}
Note that each item $i\in[n]$ is of exactly one of these types. We now distinguish whether $\pi^{(0)}=\pi^{(1)}$, $\pi^{(0)}<\pi^{(1)}$, or $\pi^{(0)}>\pi^{(1)}$ holds.
\begin{enumerate}
\item Let us assume that $\pi^{(0)}=\pi^{(1)}$, i.e., there are only item types $\mathcal{A}$ and $\mathcal{B}$. Note that we can assume items in $\mathcal{A}$ are always packed using $\pmb{y}^{(1)}$ in sorted order. We claim that there exists an optimal solution where $y^{(1)}_i = 0$ for all items $i\in\mathcal{B}$. To see this, let us assume that $y^{(1)}_i=1$ for some $i\in\mathcal{B}$. Then there exist some $j\in\mathcal{B}$ with $y^{(0)}_j = 1$. Hence, setting either $x_i=1$ or $x_j=1$ in combination with $y^{(1)}_i=0$ and $y^{(0)}_j = 0$ does not give a worse objective value. We furthermore claim that items of type $\mathcal{B}$ are always packed using first $x_i=1$ in sorted order, and then $y^{(0)}_i=1$ in sorted order (i.e., packing with $\pmb{x}$ has preference). Let us assume that this is not the case and there exist $i<j$ with $y^{(0)}_i=1$ and $x_j=1$. Then using $y^{(0)}_j=1$ and $x_i=1$ instead gives the same objective value with respect to $\pmb{x}+\pmb{y}^{(0)}$, but an objective value that is not worse with respect to $\pmb{x}+\pmb{y}^{(1)}$.

In summary, we only need to guess the cardinality of $|\mathcal{A}|$. For each such value, an optimal solution can be constructed in $O(n)$ time, giving a total time of $O(n^{2})$.

\item Let us assume that $\pi^{(0)} < \pi^{(1)}$, i.e., there are only items of types $\mathcal{A}$, $\mathcal{B}$ and $\mathcal{C}$. Using similar arguments as in the previous case, we find that items in $\mathcal{A}$ are packed in order using $\pmb{y}^{(1)}$, items in $\mathcal{B}$ are packed in order using first $\pmb{x}$ and then $\pmb{y}^{(0)}$, and items in $\mathcal{C}$ are packed in order using only $\pmb{x}$. Guessing the cardinality of items packed in $\mathcal{A}$ and $\mathcal{C}$, we can construct an optimal solution in $O(n^3)$ time.

\item Let us assume that  $\pi^{(0)} > \pi^{(1)}$, i.e., there are only items of types $\mathcal{A}$, $\mathcal{B}$ and $\mathcal{D}$. Using similar arguments as in the first case, items in $\mathcal{A}$ are packed in order using $\pmb{y}^{(1)}$, items in $\mathcal{B}$ are packed in order using first $\pmb{x}$ and then $\pmb{y}^{(0)}$. Items in $\mathcal{D}$ can be packed by both $\pmb{y}^{(0)}$ and $\pmb{y}^{(1)}$, and can be assumed to be packed in order by each vector. Let us guess $\sum_{i\in[n]} x_i = \sum_{i\in\mathcal{B}} x_i$. We then construct an optimal $\pmb{y}^{(0)}$ solution by packing remaining items in $\mathcal{B}$ and $\mathcal{D}$ in order, and construct an optimal $\pmb{y}^{(1)}$ solution by packing remaining items in $\mathcal{A}$ and $\mathcal{D}$ in order. Hence, this setting can be solved in $O(n^{2})$ time.
\end{enumerate}

We can conclude that problem~\eqref{special2} can be solved in $O(n^3)$ time as claimed.
\end{proof}

\section{Experiments}
\label{sec:experiments}

In the previous sections we analyzed the complexity of our two-stage approach, identifying easy and hard problem cases. The purpose of this section is to assess the benefit of our approach in comparison to classic one-stage robust optimization, i.e., what is the advantage of an additional stage that allows the decision maker to observe the costs of the items chosen in the first stage before completing the solution? To this end, we perform computational experiments which compare the objective value of both settings.

\subsection{Setup}

We focus on instances of the \textsc{selection} problem, for which we need to determine the number of items $n$, the number of items that must be selected $p$, the uncertainty parameter $\Gamma$, as well as lower and upper bounds on first- and second-stage costs. For these experiments, we choose a fixed size $n=20$ and vary both $p$ and $\Gamma$.

To generate lower and upper bounds on item costs, we uniformly sample three random integers in $\{1,\ldots,100\}$ per item $i$. Let $v^{(1)} \le v^{(2)} \le v^{(3)}$ be these values after sorting them. We then set $\underline{c}_i = \underline{d}_i = v^{(1)}$, $\overline{c}_i = v^{(2)}$, and $\overline{d}_i = v^{(3)}$. This setting reflects a higher degree of uncertainty for decisions that lie in the future. We generate 50 instances this way. When changing parameters $p$ and $\Gamma$, the same set of instances is solved for better comparability.

Each instance is solved with two methods. The first reflects the one-stage setting, where a complete solution must be bought in one step, i.e., a static policy is used which may not depend on the reaction of the adversary. Formally, this corresponds to the problem
\[ \min_{\pmb{x}\in\X} \min_{\pmb{y}\in\X(\pmb{x})} \max_{\pmb{c}\in\cU} \max_{\pmb{d}\in\cU(\pmb{c})} \pmb{c}^t\pmb{x} + \pmb{d}^t \pmb{y}. \]
In our case, by dualizing the inner adversarial problem with budgeted uncertainty, this is equivalent to the following mixed-integer program:
\begin{align*}
\min\ & \sum_{i\in[n]} (\underline{c}_i x_i + \underline{d}_i y_i) + \Gamma\pi + \sum_{i\in[n]} \rho_i \\
\text{s.t. } & \sum_{i\in[n]} (x_i + y_i) \ge p \\
& \pi + \rho_i \ge (\overline{c}_i-\underline{c}_i)x_i + (\overline{d}_i - \underline{d}_i)y_i & \forall i\in[n] \\
& x_i + y_i \le 1 & \forall i\in[n] \\
& x_i, y_i \in\{0,1\} & \forall i\in[n]
\end{align*}
The second method is our two-stage problem \textsc{Rob} with discrete budgeted uncertainty, modeled through the compact formulation~\eqref{pd7}. Note that the optimal objective value resulting from \textsc{Rob} with discrete budgeted uncertainty $R_{2D}$ is less or equal to the objective value of \textsc{Rob} with continuous budgeted uncertainty $R_{2C}$, which is in turn less or equal to the objective value of the one-stage approach $R_1$. Hence, the ratio $gap = R_1 / R_{2D} -1$ is an upper bound to the corresponding ratio $ R_1 / R_{2C} -1$. We use this definition of $gap$ to measure the difference between the one-stage and two-stage approaches.

All instances are solved with CPLEX 12.8. on a virtual server with Intel Xeon Gold 5220 CPU running at 2.20GHz. Each process is single-threaded. We use a time limit of 300 seconds per instance. As the computation times for solving the one-stage approach are small ($<0.1$ seconds across all instances), we provide CPLEX with the one-stage solution as a warmstart when solving the two-stage model.

\subsection{Results}

We show average values of $gap$ over the 50 test instances using different values for $p$ and $\Gamma$ in Figure~\ref{fig:gap} and median computation times of the two-stage model in Figure~\ref{fig:times}.

\begin{figure}[p]
\begin{center}
\includegraphics[width=0.6\textwidth]{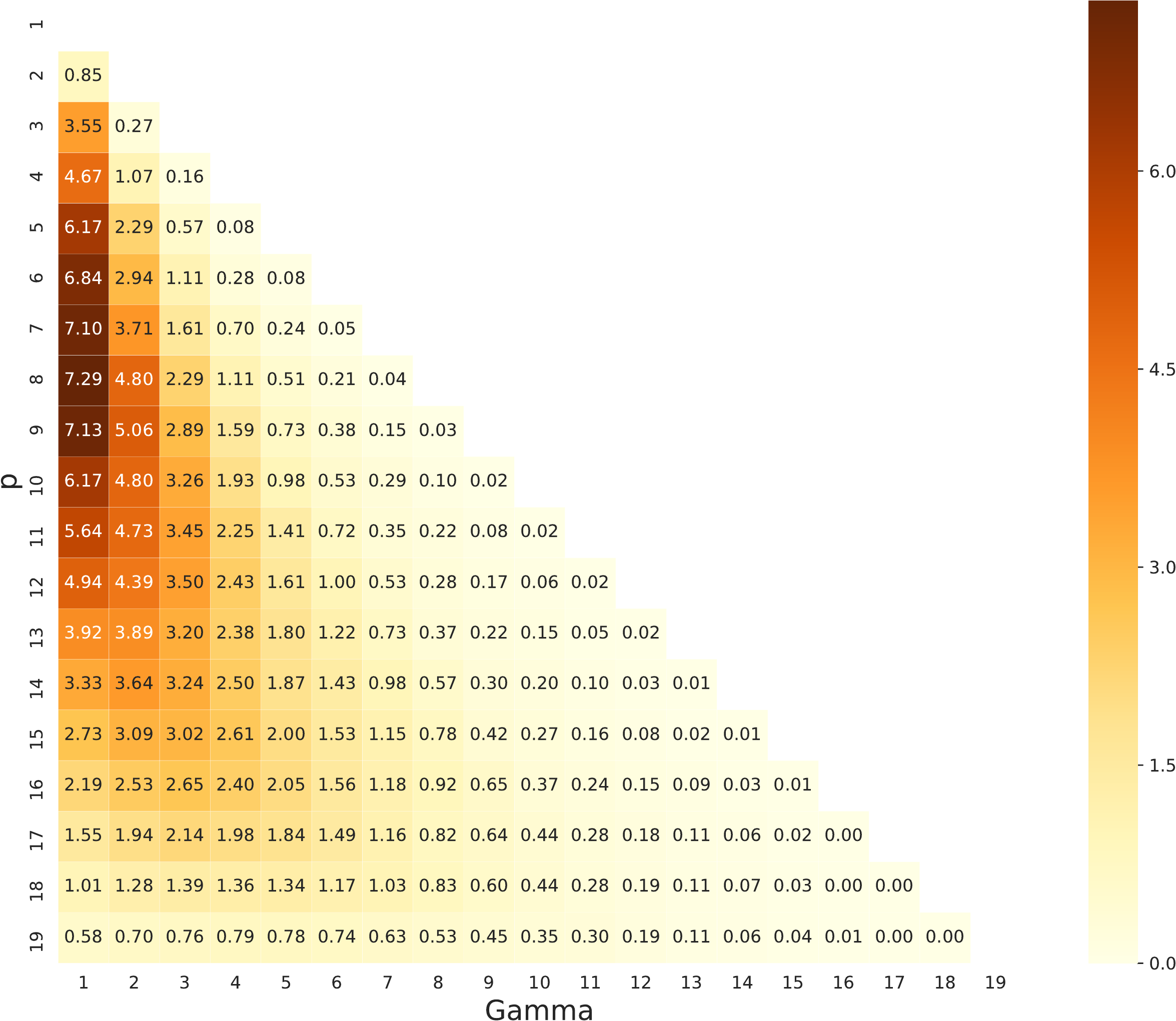}
\end{center}
\caption{Average values of $gap$.}\label{fig:gap}
\end{figure}

\begin{figure}[p]
\begin{center}
\includegraphics[width=0.6\textwidth]{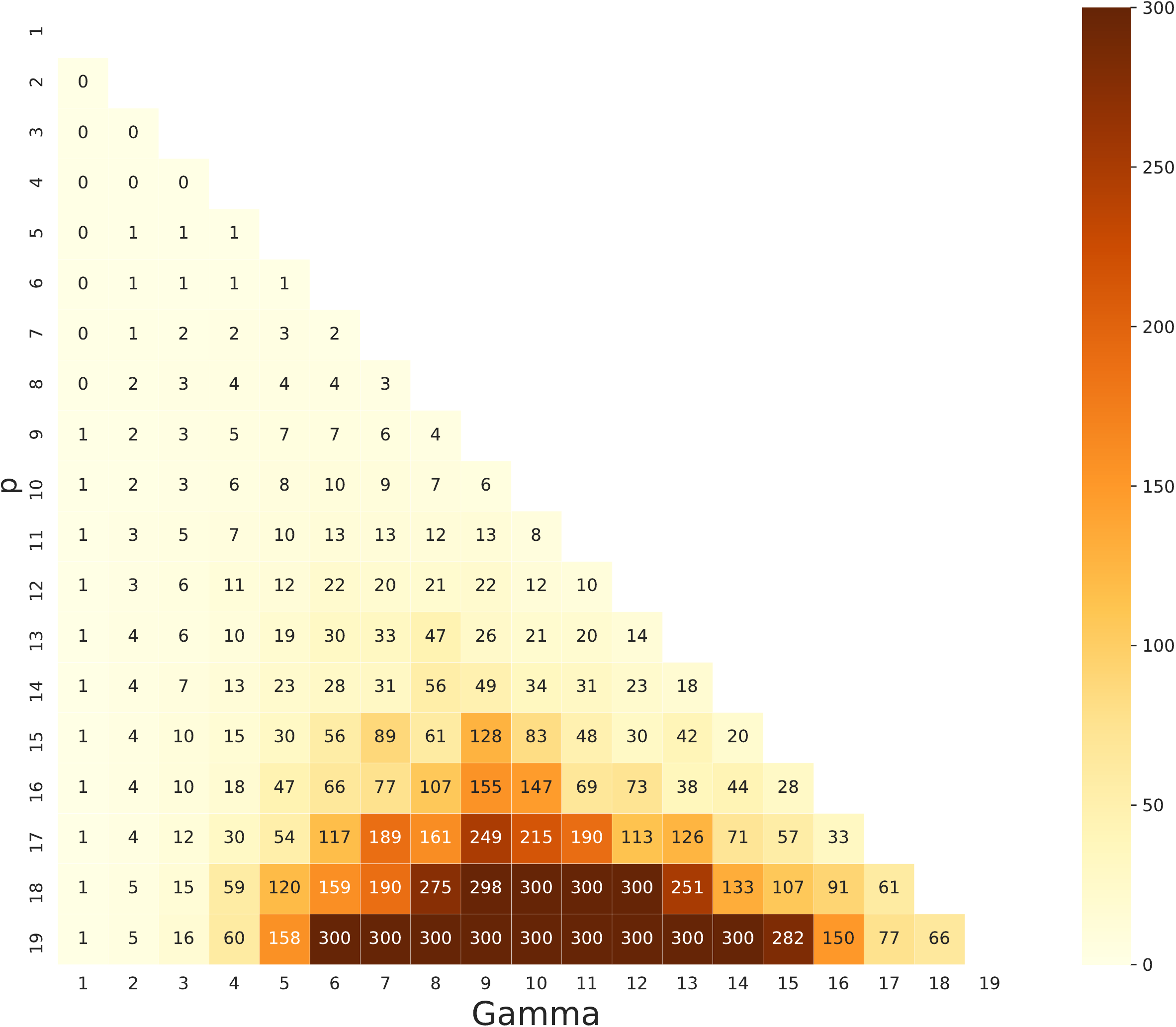}
\end{center}
\caption{Median computation times for the two-stage approach.}\label{fig:times}
\end{figure}

We first consider average $gap$ values. In this figure, darker colors indicate higher average $gap$ values and thus a larger benefit of using the two-stage model instead of the one-stage model. All values in Figure~\ref{fig:gap} are given in percent. Combinations with $p=1$, $p=20$, or $p=\Gamma$ are not shown, because in these cases we know a priori that there is no gap.

It can be seen that the average value of $gap$ depends on the combination of $p$ and $\Gamma$ parameters. For $p=8$ and $\Gamma=1$, a maximum value of $7.29\%$ is reached. Further relatively high values are in the region of small $\Gamma$ and value of $p$ up to around 14. To provide an intuition why the two-stage setting is particularly beneficial in this region, consider $\Gamma=1$. If the adversary decides to spend this budget on the first-stage solution, then the decision maker can choose items according to their nominal costs in the second stage, i.e., completely ignore uncertainty. In other words, the information reveal by the adversary is particularly large. For increasing values of $p$, the additional costs caused by the adversary become less significant in comparison to the total costs of the solution, and the $gap$ value becomes smaller again. Instead, slightly higher values of $\Gamma$ give slightly higher $gap$ values. In comparison, if $\Gamma$ is close to $p$, then the adversary can attack nearly all items. Accordingly, the information which items are attacked becomes less valuable to the decision maker.

We now consider median computation times as presented in Figure~\ref{fig:times}, where we rounded values to the next integer second. Recall that a time limit of 300 seconds was used. In the nominal \textsc{selection} problem, there are ${n}\choose{p}$ many feasible solutions, which may give the impression that problems with $p\approx n/2$ are the hardest to solve. This is not the case; instead, we note that the larger the value of $p$, the higher the computation times, in particular in the region where $\Gamma \approx p/2$. A possible explanation for this behavior is that in the two-stage setting, we also must make a decision when to buy an item (whether in the first stage or in the second stage). Hence, with more items that we need to select, more such decisions arise.

On the one hand, the relatively high median computation times for large values of $p$ (recall that each one-stage problem requires less than $0.1$ seconds to solve) show a disadvantage of the two-stage setting, which is the increased problem complexity. On the other hand, in combination with Figure~\ref{fig:gap} we note that the hardest problems are also those where the two-stage setting is least relevant. That is, problems with large parameter $p$ may be hard to solve, but in such cases the two-stage setting would not be useful in the first place. In the parameter range where the two-stage setting is most beneficial, computation times remain small, which indicates that the approach may also be useful in practice.

\section{Conclusions}
\label{sec:conclusions}

In this paper we extended the notion of two-stage robust problems by introducing two-stage uncertainty, where the adversary has the opportunity to react to the second stage of the decision maker. In particular, classic two-stage problems of the form
\[ \min_{\pmb{x}\in\X} \max_{\pmb{c}\in\cU} \min_{\pmb{y}\in\X(\pmb{x})} \pmb{C}^t\pmb{x} + \pmb{c}^t\pmb{y} \]
imply that the adversary needs to make a forward-facing decision that affects item choices not yet made, whereas in our setting, we consider a more natural extension of the one-stage problem
\[ \min_{\pmb{x}\in\X} \max_{\pmb{c}\in\cU} \pmb{c}^t\pmb{x} \]
where the adversary remains backward-facing, i.e., affects a decision already fixed. Analyzing two-stage continuous budgeted uncertainty sets, we show that a similar decomposition as in the classic paper \cite{bertsimas2003robust} is possible, resulting in a range of combinatorial optimization problems that remain solvable in polynomial time. This result does not hold for two-stage discrete budgeted sets, where already simple combinatorial optimization problems become NP-hard.

This paper makes a first sortie into a new type of robust optimization problems, where many further interesting problems remain to be considered. An open question is whether problem~\eqref{p8} can be solved in polynomial time for the case of spanning tree. Furthermore, the special case where $\underline{\pmb{c}}=\underline{\pmb{d}}$ and $\overline{\pmb{c}}=\overline{\pmb{d}}$ can still be solved for some problems with $\Gamma=1$.
It is an open problem if this remains possible for higher values of $\Gamma$.
Also note that the hardness proofs for two-stage discrete budgeted uncertainty do not give consequences on inapproximability; it is therefore interesting to consider if approximation algorithms are possible. Finally, other uncertainty sets can be considered. Recall that one-stage problem is NP-hard for all relevant combinatorial optimization problems with discrete uncertainty $\cU_D = \{\pmb{c}^{(1)}, \dots, \pmb{c}^{(N)}\}$.
We conjecture that the two-stage setting proposed here with a variant of two-stage discrete uncertainty moves up one level in the complexity hierarchy and problems become $\Sigma^p_2$-hard in many cases. 

\subsubsection*{Acknowledgements.}
Lasse Wulf acknowledges the support of the Austrian Science Fund (FWF): W1230. We thank the reviewers for their constructive feedback that further improved the paper quality.

\newcommand{\etalchar}[1]{$^{#1}$}

\appendix


\begin{thebibliography}{BTGGN04}

\bibitem[BD16]{bertsimas2016multistage}
Dimitris Bertsimas and Iain Dunning.
\newblock Multistage robust mixed-integer optimization with adaptive
  partitions.
\newblock {\em Operations Research}, 64(4):980--998, 2016.

\bibitem[BGKS14]{bruns2014robust}
Florian Bruns, Marc Goerigk, Sigrid Knust, and Anita Sch{\"o}bel.
\newblock Robust load planning of trains in intermodal transportation.
\newblock {\em OR spectrum}, 36(3):631--668, 2014.

\bibitem[BMV10]{bohle2010robust}
Carlos Bohle, Sergio Maturana, and Jorge Vera.
\newblock A robust optimization approach to wine grape harvesting scheduling.
\newblock {\em European Journal of Operational Research}, 200(1):245--252,
  2010.

\bibitem[BS03]{bertsimas2003robust}
Dimitris Bertsimas and Melvyn Sim.
\newblock Robust discrete optimization and network flows.
\newblock {\em Mathematical Programming}, 98(1):49--71, 2003.

\bibitem[BS04]{bertsimas2004price}
Dimitris Bertsimas and Melvyn Sim.
\newblock The price of robustness.
\newblock {\em Operations Research}, 52(1):35--53, 2004.

\bibitem[BT06]{bertsimas2006robust}
Dimitris Bertsimas and Aur{\'e}lie Thiele.
\newblock A robust optimization approach to inventory theory.
\newblock {\em Operations research}, 54(1):150--168, 2006.

\bibitem[BTEGN09]{ben2009robust}
Aharon Ben-Tal, Laurent El~Ghaoui, and Arkadi Nemirovski.
\newblock {\em Robust optimization}.
\newblock Princeton University Press, 2009.

\bibitem[BTGGN04]{ben2004adjustable}
Aharon Ben-Tal, Alexander Goryashko, Elana Guslitzer, and Arkadi Nemirovski.
\newblock Adjustable robust solutions of uncertain linear programs.
\newblock {\em Mathematical Programming}, 99(2):351--376, 2004.

\bibitem[CG21]{chassein2020complexity}
Andr{\'e} Chassein and Marc Goerigk.
\newblock On the complexity of min--max--min robustness with two alternatives
  and budgeted uncertainty.
\newblock {\em Discrete Applied Mathematics}, 296:141--163, 2021.

\bibitem[CGKZ18]{chassein2018recoverable}
Andr{\'e} Chassein, Marc Goerigk, Adam Kasperski, and Pawe{\l} Zieli{\'n}ski.
\newblock On recoverable and two-stage robust selection problems with budgeted
  uncertainty.
\newblock {\em European Journal of Operational Research}, 265(2):423--436,
  2018.

\bibitem[DEMM18]{durr2018scheduling}
Christoph D{\"u}rr, Thomas Erlebach, Nicole Megow, and Julie Mei{\ss}ner.
\newblock Scheduling with explorable uncertainty.
\newblock In {\em 9th Innovations in Theoretical Computer Science Conference
  (ITCS 2018)}. Schloss Dagstuhl-Leibniz-Zentrum fuer Informatik, 2018.

\bibitem[DEMM20]{durr2020adversarial}
Christoph D{\"u}rr, Thomas Erlebach, Nicole Megow, and Julie Mei{\ss}ner.
\newblock An adversarial model for scheduling with testing.
\newblock {\em Algorithmica}, 82(12):3630--3675, 2020.

\bibitem[EHK{\etalchar{+}}08]{erlebach2008computing}
Thomas Erlebach, Michael Hoffmann, Danny Krizanc, Mat{\'u}s Mihal'{\'a}k, and
  Rajeev Raman.
\newblock Computing minimum spanning trees with uncertainty.
\newblock In {\em STACS 2008}, pages 277--288. IBFI Schloss Dagstuhl, 2008.

\bibitem[FMO{\etalchar{+}}07]{feder2007computing}
Tom{\'a}s Feder, Rajeev Motwani, Liadan O'Callaghan, Chris Olston, and Rina
  Panigrahy.
\newblock Computing shortest paths with uncertainty.
\newblock {\em Journal of Algorithms}, 62(1):1--18, 2007.

\bibitem[FMP{\etalchar{+}}00]{feder2000computing}
Tomas Feder, Rajeev Motwani, Rina Panigrahy, Chris Olston, and Jennifer Widom.
\newblock Computing the median with uncertainty.
\newblock In {\em Proceedings of the thirty-second annual ACM symposium on
  Theory of computing}, pages 602--607, 2000.

\bibitem[GDT15]{goerigk2015two}
Marc Goerigk, Kaouthar Deghdak, and Vincent T’Kindt.
\newblock A two-stage robustness approach to evacuation planning with buses.
\newblock {\em Transportation Research Part B: Methodological}, 78:66--82,
  2015.

\bibitem[GGI{\etalchar{+}}15]{goerigk2015robust}
Marc Goerigk, Manoj Gupta, Jonas Ide, Anita Sch{\"o}bel, and Sandeep Sen.
\newblock The robust knapsack problem with queries.
\newblock {\em Computers \& Operations Research}, 55:12--22, 2015.

\bibitem[GGM06]{goel2006asking}
Ashish Goel, Sudipto Guha, and Kamesh Munagala.
\newblock Asking the right questions: Model-driven optimization using probes.
\newblock In {\em Proceedings of the twenty-fifth ACM SIGMOD-SIGACT-SIGART
  symposium on Principles of database systems}, pages 203--212, 2006.

\bibitem[GH21]{goerigk2020multistage}
Marc Goerigk and Michael Hartisch.
\newblock Multistage robust discrete optimization via quantified integer
  programming.
\newblock {\em Computers \& Operations Research}, 135(105434), 2021.

\bibitem[GJ79]{garey1979computers}
Michael~R Garey and David~S Johnson.
\newblock {\em Computers and intractability}.
\newblock W.~H.~Freeman, 1979.

\bibitem[GKZ20]{goerigk2020two}
Marc Goerigk, Adam Kasperski, and Pawe{\l} Zieli{\'n}ski.
\newblock Two-stage combinatorial optimization problems under risk.
\newblock {\em Theoretical Computer Science}, 804:29--45, 2020.

\bibitem[GLW20]{goerigk2020recoverable}
Marc Goerigk, Stefan Lendl, and Lasse Wulf.
\newblock Recoverable robust representatives selection problems with discrete
  budgeted uncertainty.
\newblock {\em arXiv preprint arXiv:2008.12727}, 2020.

\bibitem[GM07]{guha2007model}
Sudipto Guha and Kamesh Munagala.
\newblock Model-driven optimization using adaptive probes.
\newblock In {\em SODA}, volume~7, pages 308--317, 2007.

\bibitem[HdL21]{halldorsson2021query}
Magn{\'u}s~M Halld{\'o}rsson and Murilo~Santos de~Lima.
\newblock Query-competitive sorting with uncertainty.
\newblock {\em Theoretical Computer Science}, 867:50--67, 2021.

\bibitem[IT21]{iwamasa}
Yuni Iwamasa and Kenjiro Takazawa.
\newblock Optimal matroid bases with intersection constraints:valuated
  matroids, {M}-convex functions, and their applications.
\newblock {\em Mathematical Programming}, online first, 2021.

\bibitem[KVKV11]{korte2011combinatorial}
Bernhard~H Korte, Jens Vygen, B~Korte, and J~Vygen.
\newblock {\em Combinatorial optimization}, volume~1.
\newblock Springer, 2011.

\bibitem[KZ16]{kasperski2016robust}
Adam Kasperski and Pawe{\l} Zieli{\'n}ski.
\newblock Robust discrete optimization under discrete and interval uncertainty:
  A survey.
\newblock In {\em Robustness analysis in decision aiding, optimization, and
  analytics}, pages 113--143. Springer, 2016.

\bibitem[KZ17a]{kasperski2017robust}
Adam Kasperski and Pawe{\l} Zieli{\'n}ski.
\newblock Robust recoverable and two-stage selection problems.
\newblock {\em Discrete Applied Mathematics}, 233:52--64, 2017.

\bibitem[KZ17b]{kasperski2017robustb}
Adam Kasperski and Pawe{\l} Zieli{\'n}ski.
\newblock Robust two-stage network problems.
\newblock In {\em Operations Research Proceedings 2015}, pages 35--40.
  Springer, 2017.

\bibitem[L{\'C}P19]{lendl2019combinatorial}
Stefan Lendl, Ante {\'C}usti{\'c}, and Abraham~P Punnen.
\newblock Combinatorial optimization with interaction costs: Complexity and
  solvable cases.
\newblock {\em Discrete Optimization}, 33:101--117, 2019.

\bibitem[Lee22]{lee2022robust}
Chungmok Lee.
\newblock A robust optimization approach with probe-able uncertainty.
\newblock {\em European Journal of Operational Research}, 296(1):218--239,
  2022.

\bibitem[MMS17]{megow2017randomization}
Nicole Megow, Julie Mei{\ss}ner, and Martin Skutella.
\newblock Randomization helps computing a minimum spanning tree under
  uncertainty.
\newblock {\em SIAM Journal on Computing}, 46(4):1217--1240, 2017.

\bibitem[NO13]{nasrabadi2013robust}
Ebrahim Nasrabadi and James~B Orlin.
\newblock Robust optimization with incremental recourse.
\newblock {\em arXiv preprint arXiv:1312.4075}, 2013.

\bibitem[VGY20]{vayanos2020robust}
Phebe Vayanos, Angelos Georghiou, and Han Yu.
\newblock Robust optimization with decision-dependent information discovery.
\newblock {\em arXiv preprint arXiv:2004.08490}, 2020.

\end{thebibliography}
\end{document}